\documentclass[12pt,a4paper]{amsart}
\usepackage[shortlabels]{enumitem}
\usepackage[margin=2.5cm]{geometry}
\usepackage{amssymb}
\usepackage{graphicx}
\usepackage{mathrsfs}
\usepackage[titletoc]{appendix}

\usepackage[usenames,dvipsnames]{xcolor}
\definecolor{darkblue}{rgb}{0.0, 0.0, 0.55}
\definecolor{bordeaux}{rgb}{0.34, 0.01, 0.1}

\usepackage[pagebackref,colorlinks,linkcolor=bordeaux,citecolor=darkblue,urlcolor=black,hypertexnames=true]{hyperref}

\newtheorem{theorem}{Theorem}[section]
\newtheorem{corollary}[theorem]{Corollary}
\newtheorem{lemma}[theorem]{Lemma}
\newtheorem{proposition}[theorem]{Proposition}
\newtheorem{thmA}{Theorem}

\theoremstyle{definition}

\newtheorem{remark}[theorem]{Remark}
\newtheorem{example}[theorem]{Example}

\numberwithin{equation}{section}

\DeclareMathOperator{\Sym}{Sym}

\DeclareMathOperator{\tr}{tr}
\DeclareMathOperator{\supp}{supp}

\DeclareMathOperator{\id}{id}

\newif\ifcomment
\commentfalse
\commenttrue

\usepackage{xspace}
\newcommand{\sigalg}{$\sigma$-algebra\xspace}

\begin{document}
\def\cA{\mathcal A}
\def\cH{\mathcal H}
\def\cK{\mathcal K}
\def\cX{\mathcal X}
\def\red{\color{red}}
\def\bl{\color{blue}}
\def\ora{\color{orange}}
\def\green{\color{green}}
\def\br{\color{brown}}
\def\la{\langle}
\def\ra{\rangle}
\def\e{{\rm e}}
\def\x{\mathbf{x}}
\def\by{\mathbf{y}}
\def\bz{\mathbf{z}}
\def\cC{\mathcal{C}}
\def\R{\mathbb{R}}
\def\C{\mathbb{C}}
\def\O{\operatorname{O}}
\def\Mbb{\mathbb{M}}
\def\Sbb{\mathbb{S}}
\newcommand*{\sbb}[1]{\operatorname{S}_{#1}(\R)}
\def\T{\mathbb{T}}
\def\N{\mathbb{N}}
\def\K{\mathbb{K}}
\def\bK{\overline{\mathbf{K}}}
\def\Q{\mathbf{Q}}
\def\M{\mathbf{H}}
\def\O{\mathbf{O}}
\def\Hbb{\mathbb{H}}
\def\P{\mathbf{P}}
\def\Z{\mathbb{Z}}
\def\A{\mathbf{A}}
\def\W{\mathbf{W}}
\def\bfone{\mathbf{1}}
\def\V{\mathbf{V}}
\def\AA{\overline{\mathbf{A}}}
\def\bL{\mathbf{L}}
\def\bS{\mathbf{S}}
\def\Y{\mathbf{Y}}
\def\G{\mathbf{G}}
\def\Bbb{\mathbb{B}}
\def\Dbb{\mathbb{D}}
\def\f{\mathbf{f}}
\def\z{\mathbf{z}}
\def\bx{\mathbf{x}}
\def\h{\mathbf{h}}
\def\u{\mathbf{u}}
\def\g{\mathbf{g}}
\def\w{\mathbf{w}}
\def\a{\mathbf{a}}
\def\q{\mathbf{q}}
\def\u{\mathbf{u}}
\def\vb{\mathbf{v}}
\def\s{\mathcal{S}}
\def\cD{\mathcal{D}}
\def\co{{\rm co}\,}
\def\cp{{\rm CP}}
\def\tg{\tilde{f}}
\def\tx{\tilde{\x}}
\def\supmu{{\rm supp}\,\mu}
\def\supnu{{\rm supp}\,\nu}
\def\m{\mathcal{M}}
\def\bR{\mathbf{R}}
\def\om{\mathbf{\Omega}}
\def\s{\mathcal{S}}
\def\k{\mathcal{K}}
\def\la{\langle}
\def\ra{\rangle}
\def\sig{\varsigma}
\def\bcK{{\mathbf{K}}}
\def\bE{{\rm E}}

\def\y{\mathtt{m}}
\def\ux{\underline x}
\def\mp{\mathscr{M}}
\def\MP{\mp[\ux]}
\def\px{\R[\ux]}
\newcommand{\prob}[1]{\mathbf{P}(#1)}
\def\d{{\,\rm d}}
\def\ve{\varepsilon}
\newcommand{\QM}[1]{\operatorname{QM}(#1)}
\newcommand{\qm}[1]{\operatorname{qm}(#1)}
\newcommand{\QQM}[1]{\widetilde{\operatorname{QM}}(#1)}
\newcommand{\qqm}[1]{\widetilde{\operatorname{qm}}(#1)}

\newcommand{\bra}[1]{\mathinner{\langle #1|}}
\newcommand{\ket}[1]{\mathinner{|#1\rangle}}
\newcommand{\braket}[2]{\mathinner{\langle #1|#2\rangle}}
\newcommand{\dyad}[1]{| #1\rangle \langle #1|}

\def\cP{\mathcal{P}}
\def\cM{\mathcal{M}}
\def\cQ{\operatorname{QM}_\sig}
\def\cN{\mathcal{N}}
\def\cF{\mathcal{F}}
\def\cE{\mathcal{E}}
\def\cB{\mathcal{B}}
\def\cL{\mathcal{L}}

\def\smileL{\overset{\smallsmile}{L}}
\def\blambda{{\boldsymbol{\lambda}}}
\def\bsigma{{\Delta}}
\def\RX{\R \langle \underline{x} \rangle}
\def\RXk{\R \langle \underline{x}(I_k) \rangle}
\def\RXonetwo{\R \langle \underline{x}(I_1 \cap I_2) \rangle}
\def\CX{\C \langle \underline{x} \rangle}
\def\TX{\fatT}
\def\KX{\K \langle \underline{x} \rangle}
\def\uX{\underline X}
\def\uY{\underline Y}
\def\uF{\underline F}
\def\mx{\langle\underline x\rangle}

\def\SymS{\Sym \fatS}
\def\ov{\overline{o}}
\def\und{\underline{o}}
\newcommand{\victor}[1]{\Vi{#1}}
\newcommand{\victorshort}[1]{\todo[inline,color=purple!30]{VM: #1}}
\newcommand{\igor}[1]{\Ig{#1}}

\setcounter{secnumdepth}{3}
\setcounter{tocdepth}{3}
\makeatletter
\newcommand{\mycontentsbox}{%
\printindex
{\centerline{NOT FOR PUBLICATION}
\addtolength{\parskip}{-2.0pt}\normalsize
\tableofcontents}}
\def\enddoc@text{\ifx\@empty\@translators \else\@settranslators\fi
\ifx\@empty\addresses \else\@setaddresses\fi
\newpage\mycontentsbox
}
\makeatother

\colorlet{commentcolour}{green!50!black}
\newcommand{\comment}[3]{%
\ifcomment%
	{\color{#1}\bfseries\sffamily(#3)%
	}%
	\marginpar{\textcolor{#1}{\hspace{3em}\bfseries\sffamily #2}}%
	\else%
	\fi%
}
\newcommand{\Ig}[1]{
	\comment{magenta}{I}{#1}
}
\newcommand{\Vi}[1]{
	\comment{blue}{V}{#1}
}
\newcommand{\jurij}[1]{
	\comment{orange}{J}{#1}
}

\newcommand{\idea}[1]{\textcolor{red}{#1(?)}}

\newcommand{\Expl}[1]{
	{\tag*{\text{\small{\color{commentcolour}#1}}}%
	}
}

\title[Sums of squares certificates for polynomial moment inequalities]{Sums of squares certificates for\\[1mm] polynomial moment inequalities}
\author{Igor Klep \and Victor Magron \and Jurij Vol\v{c}i\v{c} }
\date{\today}
\address{Igor Klep: Faculty of Mathematics and Physics, Department of Mathematics,  University of Ljubljana \& Institute of Mathematics, Physics and Mechanics, Ljubljana, Slovenia}
\email{igor.klep@fmf.uni-lj.si}
\thanks{IK was supported by the 
Slovenian Research Agency grants 
J1-50002, N1-0217, J1-3004 and P1-0222.
}
\address{Victor Magron: LAAS-CNRS \& Institute of Mathematics from Toulouse, France}
\email{vmagron@laas.fr}
\thanks{VM was supported by the EPOQCS grant funded by the LabEx CIMI (ANR-11-LABX-0040), the FastQI grant funded by the Institut Quantique Occitan, the PHC Proteus grant
46195TA, the European Union’s Horizon 2020 research and innovation programme under the Marie Sk{\l}odowska-Curie Actions, grant agreement 813211 (POEMA), by the AI Interdisciplinary Institute ANITI funding, through the French ``Investing for the Future PIA3'' program under the Grant agreement n${}^\circ$ ANR-19-PI3A-0004 as well as by the National Research Foundation, Prime Minister’s Office, Singapore under its Campus for Research Excellence and Technological Enterprise (CREATE) programme.}
\address{Jurij Vol\v{c}i\v{c}: Department of Mathematics, Drexel University, Pennsylvania}
\email{jurij.volcic@drexel.edu}
\thanks{JV was supported by the NSF grant DMS-1954709.}
\date{}

\begin{abstract}
This paper introduces and develops the algebraic framework of moment polynomials, which are polynomial expressions in commuting variables and their formal mixed moments. Their positivity and optimization over probability measures supported on semialgebraic sets and subject to moment polynomial constraints is investigated. A positive solution to Hilbert's 17th problem for pseudo-moments is given. On the other hand, moment polynomials positive on actual measures are shown to be sums of squares and formal moments of squares up to arbitrarily small perturbation of their coefficients. When only measures supported on a bounded semialgebraic set are considered, a stronger algebraic certificate for moment polynomial positivity is derived. This result gives rise to a converging hierarchy of semidefinite programs for moment polynomial optimization. Finally, as an application, two open nonlinear Bell inequalities from quantum physics are settled.
\end{abstract}



\maketitle

\section{Introduction}
\label{sec:intro}

Consider two groups of independent variables, $x_1,\dots,x_n$ and $\y_{i_1,\dots,i_n}$ for $i_1,\dots i_n\in\N_0$.
Here, the latter are viewed as formal mixed moments, also denoted $\y_{i_1,\dots,i_n}=\y(x_1^{i_1}\cdots x_n^{i_n})$, which in the presence of a probability measure $\mu$ on $\R^n$ evaluate as $\int x_1^{i_1}\cdots x_n^{i_n} \d \mu$. 
This paper focuses on the class of  \emph{moment polynomials}, i.e., polynomials in $x_1,\dots,x_n$ and their formal moments, and their inequalities.
A very rudimentary instance of a moment polynomial inequality is that every random variable $X$ has a nonnegative variance, 
\[
	{\rm Var}(X)=\bE(X^2)-\bE(X)^2 = \bE
\big((X-\bE(X))^2 \big)\ge0;
	\] in the language of this paper, we say that the moment polynomial $\y(x_1^2)-\y(x_1)^2$ is nonnegative.
Problems involving moment polynomial inequalities and optimization arise in various fields. 
For example, concentration inequalities in moments provide bounds on deviation of a random variable, and the search for such inequalities has been a flourishing area of probability theory \cite{boucheron05,BP05,mackey14}.
In statistics,
distributions of structural parameters are partially identified by moment inequalities \cite{liao10};
similarly, moment inequality constraints can be used as a basis for estimation and inference in partially identified behavioral models in economics \cite{pakes15} and industrial organization \cite{kline21}.
Another instance of moment polynomial inequalities arises in operator theory \cite{curto2010}, where they characterize hyponormality of operators.
Moment optimization is used to study partial differential equations:
in \cite{henrion23}, solving a heat equation with a nonlinear perturbation is formulated as a linear optimization problem on moments, while \cite{fantuzzi} applies polynomial optimization to verify integral inequalities arising from Lyapunov analysis of fluid flows.
Lastly, nonlinear Bell inequalities in quantum information theory are moment inequalities that certify nonlocality in quantum networks \cite{PHBB,TGB21,tavakoli22}. 

To systematically approach moment polynomial inequalities,
it is natural to start from the theory built around their moment-free analogs, namely real algebraic geometry \cite{marshallbook}. The cornerstone of real algebraic geometry are sums of squares certificates for nonnegative polynomials. Artin's solution of Hilbert's 17th problem characterizes nonnegative polynomials on $\R^n$ in terms of sums of squares and denominators, and Putinar's Positivstellensatz \cite{Putinar1993positive} describes polynomials positive on compact semialgebraic sets in $\R^n$.
The latter was groundbreakingly applied to polynomial optimization in \cite{Las01sos}, resulting in Lasserre's hierarchy, based on semidefinite programming. This hierarchy yields a sequence of nondecreasing lower bounds converging to the global infimum of a polynomial over a compact semialgebraic set.
Positive polynomials also play a crucial role in functional analysis and measure theory through moment problems \cite{schmbook}. The duality between polynomials positive on a semialgebraic set $K$ and measures supported on $K$ connects sums of squares certificates with necessary conditions for solvability of the moment problem on $K$.
The monographs \cite{lasserre2009moments,henrion2020moment} present many applications of the moment problem, Lasserre's hierarchy and its variations.
More recent developments in this field concern nonlinear expressions in moments, and the infinite-dimensional moment problem.
In \cite{blekherman2022}, techniques of tropical geometry are applied to nonnegative polynomials and moment problems, resulting in classification of moment \emph{binomial} inequalities.
In the recent work \cite{henrion23}, nonlinear partial differential equations are formulated  as moment problems for measures supported on infinite-dimensional vector spaces,
and then results about the infinite-dimensional moment problem in nuclear spaces \cite{infusino2014,infusino2023moment} are leveraged to derive converging approximations to solutions of differential equations. 

In the noncommutative setting, the Helton-McCullough Positivstellensatz \cite{Helton04} leads to similar methods for optimizing eigenvalues of polynomials in matrix or operator variables \cite{burgdorf16}. 
The famous Navascu\'es-Pironio-Ac\'in hierarchy \cite{navascues2008convergent} yields bounds over the maximal violation levels of linear Bell inequalities, which also relates to the {quantum moment problem} \cite{doherty2008quantum}. 
Motivated by the more difficult study of nonlinear Bell inequalities \cite{PHBB} for correlations in quantum networks \cite{tavakoli22}, the three authors have recently proposed two nonlinear extensions to optimization problems over trace \cite{KMV} and state polynomials \cite{KMVW}, derived from Positivstellens\"atze for polynomials in noncommuting variables and formal traces or states of their products. In this paper, we let noncommutative real algebraic geometry offer a new perspective on commutative problems involving moment polynomials.  

\subsection*{Moment polynomials}

This paper 
investigates positivity and optimization of moment polynomials
subject to
polynomial relations between the problem variables $x_j$ and their formal mixed moments.
For example,
$$f=\y(x_1x_2^3)x_1x_2-\y(x_1^2)^3x_2^2+x_2-\y(x_2)\y(x_1x_2)-2$$
is a moment polynomial; at a probability measure $\mu$ on $\R^2$ with fourth order moments and a pair $(X_1,X_2)\in\R^2$, $f$ evaluates as
$$f\big(\mu,(X_1,X_2)\big)=X_1X_2\int x_1x_2^3 \d \mu -X_2^2\left(\int x_1^2\d \mu\right)^3+X_2-\int x_2\d \mu\int x_1x_2\d \mu-2.$$
A moment polynomial without freely occurring $x_j$, e.g. $\y(x_1^2x_2^2)-\y(x_1)^4+\y(x_1)\y(x_2)\y(x_1x_2)$, is called \emph{pure}.
The algebra of pure moment polynomials is denoted by $\mp$, and the algebra of moment polynomials is denoted by $\MP$. There is a natural $\mp$-linear map $\y:\MP\to\mp$ that corresponds to formal integration.

To study constrained positivity of moment polynomials, let $S_1\subseteq\px$ and $S_2\subseteq\mp$ be collections of constraints. Let $K(S_1)$ be the set of points $\uX\in\R^n$ such that all polynomials in $S_1$ are nonnegative at $\uX$. 
Let $\prob{K(S_1)}$ be the set of all Borel probability measures supported on $K(S_1)$, and let $\bcK(S_1,S_2)$ be the set of measures  $\mu\in\prob{K(S_1)}$ such that all pure moment polynomials in $S_2$ are nonnegative at $\mu$.
Adapting a standard notion from real algebra \cite{marshallbook}, we define the \emph{quadratic module} $\QM{S_1,S_2}\subseteq\MP$ as the convex hull of
$$\Big\{f^2\,\y(g^2s),\ f^2t
\colon
s\in S_1\cup\{1\},\ t\in S_1\cup S_2,\ f,g\in\MP
\Big\}.$$
Elements of $\QM{S_1,S_2}$ are clearly nonnegative on $\bcK(S_1,S_2)\times K(S_1)$.
This paper addresses the converse, and provides certificates for moment polynomial positivity on $\bcK(S_1,S_2)\times K(S_1)$ in terms of $\QM{S_1,S_2}$.

\subsection*{Main results}

The first positivity certificate applies to archimedean quadratic modules. Here, $\QM{S_1,S_2}$ is \emph{archimedean} if $N-x_1^2-\cdots-x_n^2\in \QM{S_1,S_2}$ for some $N\in\N$. Note that the constrained set $K(S_1)$ is bounded in this instance. Conversely, if $K(S_1)$ is contained in a ball of radius $R$, we may add $R^2-x_1^2-\cdots-x_n^2$ to $S_1$ to obtain an archimedean quadratic module without shrinking $\bcK(S_1,S_2)\times K(S_1)$.

\begin{thmA}[Theorem \ref{t:arch}]\label{ta:a}
If $\QM{S_1,S_2}$ is archimedean, the following statements are equivalent for $f\in\MP$:
\begin{enumerate}[\rm (i)]
    \item $f\ge0$ on $\bcK(S_1,S_2)\times K(S_1)$;
    \item $f+\ve\in\QM{S_1,S_2}$ for every $\ve>0$.
\end{enumerate}
\end{thmA}

Theorem \ref{ta:a} is proved using results from real algebra and the solution of the moment problem for compactly supported measures.
Analogously to Lasserre's hierarchy \cite{Las01sos} leveraging Putinar's Positivstellensatz \cite{Putinar1993positive} in polynomial optimization, we utilize Theorem \ref{ta:a} to derive a procedure for solving the moment polynomial optimization problem
$$\text{minimize }f(\mu,\uX) 
\text{ subject to } \uX\in K(S_1)
\text{ and } \mu\in\bcK(S_1,S_2)$$
based on semidefinite programming.

\begin{thmA}[Corollary \ref{c:hier2}]\label{ta:b}
Let $\QM{S_1,S_2}$ be archimedean and $f$ a moment polynomial. The Positivstellensatz-induced hierarchy of semidefinite
programs produces a nondecreasing sequence converging to the infimum of $f$ on $\bcK(S_1,S_2)\times K(S_1)$.
\end{thmA}

We apply Theorem \ref{ta:b} to moment polynomial optimization problems from quantum information theory. 
Two nonlinear Bell inequalities conjectured in \cite{PHBB} and \cite{TGB21,tavakoli22} are established.
For example, our optimization scheme allows us to solve the following problem:
\begin{equation}\label{e:Z=0intro}
\begin{split}
&\sup \
\frac13 \sum_{i\in\{1,2,3\} } \Big(\y(x_{i+3}x_{i+6})-\y(x_ix_{i+3})\Big)-\sum_{\{i,j,k\}=\{1,2,3\} } \y(x_ix_{j+3}x_{k+6})\\
&\text{subject to }\\
& \y(x_1^{k_1}x_2^{k_2}x_3^{k_3}x_7^{k_4}x_8^{k_5}x_9^{k_6})=
\y(x_1^{k_1}x_2^{k_2}x_3^{k_3})\y(x_7^{k_4}x_8^{k_5}x_9^{k_6}) \quad \text{for } k_i\in\{0,1\}, \\
&x_j^2=1 \text{ and }\y(x_j)=0 \quad \text{for }
j\in \{0,\dots,9\}, \\
&\y(x_ix_{j+3})=\y(x_{i+3}x_{j+6})=0\quad \text{for } i,j\in\{1,2,3\},\ i\neq j, \\
&\y(x_ix_{j+3}x_{k+6})=0\quad \text{for }i,j,k\in\{1,2,3\},\ |\{i,j,k\}|\le2.
\end{split}
\end{equation}
In Subsection \ref{sss:biloc} it is shown that the solution of \eqref{e:Z=0intro} is 4, attained by certain binary variables and the uniform measure on 16 points, thereby answering a question in \cite{TGB21,tavakoli22}.

Next, we address certificates for moment polynomial positivity subject to constraint sets $S_1$ and $S_2$ without the archimedean assumption. In particular, we aim to describe everywhere nonnegative moment polynomials (i.e., $S_1=S_2=\emptyset$). 
In this particular case, one might first consider an analog of Hilbert's 17th problem for moment polynomials (H17): if $f\in\MP$ is nonnegative on $\prob{\R^n}\times \R^n$, can we write it as a quotient of sums of products of elements of the form $f^2$ and $\y(f^2)$ for $f\in\MP$? 
It turns out that the answer to this question is negative (cf.~ Example \ref{ex:h17}). 
More precisely, the algebraic certificate in (H17) characterizes the strictly smaller class of moment polynomials that are nonnegative under \emph{pseudo-moment evaluations}. A pseudo-moment evaluation is a homomorphism $\varphi:\MP\to\R$ satisfying $\varphi(\y(p^2))\ge0$ for all $p\in\px$.

\begin{thmA}[Theorem \ref{t:h17pseudo}]\label{ta:c}
The following statements are equivalent for $f\in\MP$:
\begin{enumerate}[\rm (i)]
    \item $\varphi(f)\ge0$ for every pseudo-moment evaluation 
    $\varphi$;
    \item $f$ is a quotient of sums of products of elements of the form $h^2$ and $\y(h^2)$ for $h\in\MP$.
\end{enumerate}
\end{thmA}

In other words, (ii) means that $f$ is a quotient of two elements from the preordering in $\MP$ generated by formal moments of squares.
The proof of Theorem \ref{ta:c} relies on the Krivine-Stengle Positivstellensatz and extensions of positive functionals.
The negative answer to (H17) for moment evaluations motivates a search for a different positivity certificate. In \cite{blekherman2022}, nonnegative moment \emph{binomials} are classified in combinatorial terms. For nonnegative (classical) polynomials in $\px$, Lasserre \cite{las06perturb} showed that they become sums of squares of polynomials after an arbitrarily small perturbation of their coefficients. Our second main result generalizes Lasserre's certificate to moment polynomials.

\begin{thmA}[Theorem \ref{t:lass}]\label{ta:d}
If $S_2$ is finite, the following statements are equivalent for $f\in\MP$:
\begin{enumerate}[\rm (i)]
    \item $f\ge0$ on $\bcK(S_1,S_2)\times K(S_1)$;
    \item for every $\ve>0$ there exists $r\in\N$ such that
    $$f+\ve \sum_{j=1}^n\left(\sum_{k=0}^r\frac{x_j^{2k}}{k!}+
    \sum_{\substack{k,\ell\in\N,\\ k\ell\le r}} \frac{\y(x_j^{2k})^\ell}{(k!)^\ell \ell!}
    \right)\in\QM{S_1,S_2}.$$
\end{enumerate}
\end{thmA}

The proof of Theorem \ref{ta:d} uses constructions and techniques from functional analysis, conic programming and duality. When Theorem \ref{ta:d} is restricted to polynomials in $\px$, it improves the approximation result in \cite{LasserreNetzer}, which was established under several restrictive assumptions.

\subsection*{Acknowledgments}
The authors thank the anonymous reviewers for careful reading, thoughtful suggestions and raising crucial points that vastly improved the presentation of the paper, Didier Henrion for the discussion on infinite dimensional-moment problems, and Scott McCullough for sharing his expertise on measure theory and unbounded operators.

\section{Preliminaries}
\label{sec:prelim}

We start by introducing the notation and terminology pertaining to moment polynomials and their evaluations, and establishing some basic facts that are used throughout the paper.
Let $\px=\R[x_1,\dots,x_n]$ be the polynomial ring in $n$ variables.
Consider the polynomial ring in countably many variables
$\mp=\R[\y_{i_1,\dots,i_n}\colon i_j\in \N_0]$ where $\y_{0,\dots,0}:=1$, and denote $\MP=\mp\otimes \px$.
Elements of $\MP$ and $\mp$ are called \emph{moment polynomials} and \emph{pure moment polynomials}, respectively.
There is a canonical unital $\mp$-linear map $\y:\MP\to\mp$ determined by $\y(x_1^{i_1}\cdots x_n^{i_n})=\y_{i_1,\dots,i_n}$.

Recalling a standard notion from real algebra \cite[Section 2.1]{marshallbook}, a subset $M$ of a commutative unital ring $A$ is a \emph{quadratic module} if $1\in M$, $M+M\subseteq M$, and $a^2M\subseteq M$ for $a\in A$.
Given $S_1\subseteq \px$ and $S_2\subseteq \mp$ let $\qm{S_1,S_2}\subseteq\mp$ be the quadratic module in $\mp$ generated by $$\{\y(f^2s)\colon s\in \{1\}\cup S_1,f\in\MP \}\cup S_2,$$
and let $\QM{S_1,S_2}\subseteq\MP$ be the quadratic module in $\MP$ generated by 
$$S_1\cup\{\y(f^2s)\colon s\in \{1\}\cup S_1,f\in\MP \}\cup S_2.$$
Note that $q^2\y(f^2s)=\y((qf)^2s)$ for $q\in\mp$ and $f,s\in\MP$.
Consequently, $\qm{S_1,S_2}$ is the convex hull of
\begin{equation}\label{e:gens1}
\y(f^2s_1), \quad q^2s_2
\end{equation}
for $s_1\in \{1\}\cup S_1$, $s_2\in S_2$, $f\in\MP$ and $q\in\mp$. Similarly, 
$\QM{S_1,S_2}$ is the convex hull of
\begin{equation}\label{e:gens2}
\quad f^2\y(g^2s), \quad f^2t
\end{equation}
for $s\in \{1\}\cup S_1$, $t\in S_1\cup S_2$ and $f,g\in\MP$.
Also, let $\QM{S_1}\subset\px$ denote the quadratic module in $\px$ generated by $S_1$.

\begin{remark}
Observe that $\qm{S_1,S_2}\subseteq \QM{S_1,S_2}\cap\mp\subseteq\y(\QM{S_1,S_2})$, and these inclusions are strict in general (the first one because of term cancellations, and the second one because of terms of the form $\y(f^2)\y(g^2)$).
Therefore, for a pure moment polynomial, membership in $\qm{S_1,S_2}$ is a stronger property than membership in $\QM{S_1,S_2}$. 
Thus when stating our results for moment polynomials and $\QM{S_1,S_2}$, we also state refinements for pure moment polynomials and $\qm{S_1,S_2}$,
and the proofs are analogous. The reason for persisting with $\qm{S_1,S_2}$ is that it leads to smaller optimization problems than $\y(\QM{S_1,S_2})$.
\end{remark}

There is a natural notion of a degree $\deg$ on $\MP$ satisfying
$$\deg x_j=1,\qquad\deg \y_{i_1,\dots,i_n}=i_1+\cdots+i_n$$
and $\deg(fg)=\deg(f)+\deg(g)$ for $f,g\in\MP$.
For $r\in\N$ let $\px_r,\mp_r,\MP_r$ be the finite-dimensional subspaces of $\px,\mp,\MP$ of elements of degree at most $r$.
Also, let
$\qm{S_1,S_2}_{2r}\subseteq\mp_{2r}$ and $\QM{S_1,S_2}_{2r}\subseteq\MP_{2r}$ be the convex hulls of elements in $\MP_{2r}$ of the form \eqref{e:gens1} and \eqref{e:gens2}, respectively.

The following lemma identifies certain non-obvious elements of $\qm{\emptyset,\emptyset}$ that are required later.
It can be viewed as a formal special case of H\"older's inequality \cite[Theorem 3.5]{rudin}.

\begin{lemma}[Symbolic H\"older's inequality]\label{l:holder}
Let $k\in\N_0$.
\begin{enumerate}[\rm (1)]
\item If $n=1$ then $\y_{2k}-\y_1^{2k}\in\qm{\emptyset,\emptyset}$.
\item If $n\in\N$ and $i_1,\dots,i_n\in\N_0$ then $\y_{2k i_1,\dots,2k i_n}-\y_{i_1,\dots,i_n}^{2k}
\in\qm{\emptyset,\emptyset}$.
\end{enumerate}
\end{lemma}

\begin{proof}
\def\r{{\rm r}}
(1) The case $k=0$ holds because $\y_0-\y_1^0=0$, and the case $k=1$ holds because $\y_2-\y_1^2=\y((\y_1-x_1)^2)$.
For $k\ge2$ and $\ell=\lceil \log_2 k\rceil$ let $a_0,\dots,a_{\ell-1}$ be recursively defined as $a_0=k$ and $a_{i+1}=\lceil \frac{a_i}{2}\rceil$. Denote $\r(a_i)=0$ if $a_i$ is even and $\r(a_i)=1$ if $a_i$ is odd. Note that $a_i+\r(a_i)=2a_{i+1}$ for $0\le i<\ell-1$, and $a_{\ell-1}=2$.
Furthermore, interpreting $\r(a_i)$ as digits in a binary expansion shows that
\begin{equation}\label{e:bin}
	2^\ell-k=\sum_{i=0}^{\ell-1}2^i\r(a_i),\qquad
	k-1=\sum_{i=0}^{\ell-1}2^i(1-\r(a_i)).
\end{equation}
We claim that
\begin{equation}\label{e:lucky}
	\y_{2k}-\y_1^{2k} = 
	k\,\y\left(\Big(\y_1^k-x_1\y_1^{k-1}\Big)^2\right)
	+\sum_{i=0}^{\ell-1}2^i\,\y\left(\Big(
	x_1^{\r(a_i)}\y_1^{k-\r(a_i)}-x_1^{a_i}\y_1^{k-a_i}
	\Big)^2\right).
\end{equation}
Indeed, the right-hand side of \eqref{e:lucky} expands as
\begin{align*}
	& k \left(\y_2\y_1^{2(k-1)}-\y_1^{2k}\right)
	+\sum_{i=0}^{\ell-1}2^i \left(
	\y_{2\r(a_i)}\y_1^{2(k-\r(a_i))}
	-2\y_{a_i+\r(a_i)}\y_1^{2k-(a_i+\r(a_i))}
	+\y_{2a_i}y_1^{2(k-a_i)}
	\right) \\
	=\,& k \left(\y_2\y_1^{2(k-1)}-\y_1^{2k}\right)
	+\sum_{i=0}^{\ell-1}2^i
	\y_{2\r(a_i)}\y_1^{2(k-\r(a_i))}
	+\y_{2a_0}\y_1^{2(k-a_0)}
	-2^\ell \y_{a_{\ell-1}+\r(a_{\ell-1})}\y_1^{2k-(a_{\ell-1}+\r(a_{\ell-1}))} \\
	=\,& k \left(\y_2\y_1^{2(k-1)}-\y_1^{2k}\right)
	+\y_{2k}-2^\ell \y_2\y_1^{2(k-1)}
	+\sum_{i=0}^{\ell-1}2^i
	\y_{2\r(a_i)}\y_1^{2(k-\r(a_i))} \\
	=\,& \y_{2k}-\y_1^{2k}
	-\left(
	(k-1)\y_1^{2k}+(2^\ell-k)\y_2\y_1^{2(k-1)}
	\right)
	+\sum_{i=0}^{\ell-1}2^i
	\y_{2\r(a_i)}\y_1^{2(k-\r(a_i))} \\
	=\,& \y_{2k}-\y_1^{2k},
\end{align*}
where the last equation holds by \eqref{e:bin}.

(2)
Given $i_1,\dots,i_n\in\N_0$, consider the homomorphism
$\Theta:\R[x_1,\y_k\colon k\in\N]\to\MP$ determined by
$\Theta(x_1)=x_1^{i_1}\cdots x_n^{i_n}$ and $\Theta(\y_k)=\y_{ki_1,\dots,ki_n}$. Then $\Theta$ intertwines with the map $\y$ (i.e., $\Theta\circ \y=\y\circ\Theta$), so applying it to $\y_{2k}-\y_1^{2k}\in\qm{\emptyset,\emptyset}$ results in $\y_{2k i_1,\dots,2k i_n}-\y_{i_1,\dots,i_n}^{2k}
\in\qm{\emptyset,\emptyset}$.
\end{proof}

\begin{remark}
Not all cases of H\"older's inequality admit a symbolic interpretation as in Lemma \ref{l:holder}. For example, $\y_{11}\y_{22}-\y_{12}^2 \notin \qm{\emptyset,\emptyset}$. Indeed, if $f_1,\dots,f_\ell\in\MP$ and $uv$ with $u\in\mp$ and $v\in [\ux]$ is a monomial in one of the $f_j$ with maximal $\deg v$, then $\y(\sum_jf_j^2)$ contains the monomial $u^2 \y(v^2)$ with a positive coefficient, and consequently $\y(\sum_jf_j^2)\neq \y_{11}\y_{22}-\y_{12}^2$.
\end{remark}

\subsection{Moment evaluations of moment polynomials}

There are two natural (and closely related) types of evaluations of moment polynomials.
For a closed (but not necessarily bounded) set $K\subseteq\R^n$ let $\prob{K}$ denote the set of Borel probability measures $\mu$ on $\R^n$ that are supported on $K$ and admit all marginal moments (that is, $\int x_j^{2k}\d \mu<\infty$ for all $j=1,\dots,n$ and all $k\in\N$). By H\"older's inequality, it follows that such measures admit all mixed moments. Note that a Borel probability measure on $\R^n$ is always a Radon measure by \cite[Theorem II.3.2]{parthasarathy2005}.
Each pair $(\mu,\uX)\in\prob{\R^n}\times \R^n$ gives rise to the homomorphism
\begin{equation}\label{e:hom1}
\MP\to\R,\quad f\mapsto f(\mu,\uX)
\end{equation}
determined by 
$$\y_{i_1,\dots,i_n}\mapsto \int x_1^{i_1}\cdots x_n^{i_n}\d \mu,\qquad
x_j\mapsto X_j.$$
Such homomorphisms are called \emph{moment evaluations}.

Let $(\cX,\Sigma,\pi)$ be a probability space, where $\Sigma$ is a \sigalg on $\cX$, and $\pi:\Sigma\to\R$ is a probability measure. For $p\in\N$ let $\cL^p(\cX,\Sigma,\pi)$ be the space of real-valued random variables $F$ on $\cX$ such that $\int |F|^p\d\pi<\infty$. There is a partial order $\succeq$ on $\cL^p(\cX,\Sigma,\pi)$, given as $f\succeq g$ if $f\ge g$ almost everywhere.
Consider the ring $\cL^\omega(\pi):=\bigcap_{p=1}^\infty \cL^p(\cX,\Sigma,\pi)$ introduced in \cite{arens}.
For $\uF=(F_1,\dots,F_n)\in \cL^\omega(\pi)^n$, all the mixed moments of $\uF$ exist, and we can define the homomorphism
\begin{equation}\label{e:hom2}
\MP\to \cL^\omega(\pi),\quad f\mapsto f[\pi,\uF]
\end{equation}
determined by 
$$\y_{i_1,\dots,i_n}\mapsto \int F_1^{i_1}\cdots F_n^{i_n}\d \pi,\qquad
x_j\mapsto F_j.$$
The restrictions of homomorphisms \eqref{e:hom1} and \eqref{e:hom2} to $\mp$ coincide (when we view $\R^n$ as a probability space with the Borel \sigalg). 
Observe that the homomorphism \eqref{e:hom2} intertwines $\y:\MP\to\mp$ and integration with respect to $\pi$. In contrast, the homomorphism \eqref{e:hom1} does not satisfy such an intertwining property.

For $S_1\subseteq \px$ and $S_2\subseteq\mp$ let
\begin{align*}
	K(S_1)&=\{\uX\in\R^n\colon p(\uX)\ge0 \text{ for all } p\in S_1\} \subseteq\R^n,\\
	\bcK(S_1,S_2)&=\{\mu\in\prob{K(S_1)}\colon s(\mu)\ge0 \text{ for all } s\in S_2\}\subseteq\prob{K(S_1)}.
\end{align*}

The following proposition indicates that evaluations \eqref{e:hom1} and \eqref{e:hom2} are essentially equivalent from the perspective of moment polynomial positivity.

\begin{proposition}\label{p:evals}
	Let $S_1\subseteq\px$ and $S_2\subseteq\mp$. The following statements are equivalent for $f\in\MP$: 
\begin{enumerate}[\rm (i)]
	\item $f(\mu,\uX)\ge0$ for all $(\mu,\uX)\in \bcK(S_1,S_2)\times K(S_1)$;
	\item $f[\pi,\uF]\succeq0$ for every probability measure $\pi$ and a random variable $\uF\in\cL^\omega(\pi)^n$ with values in $K(S_1)^n$ such that $s[\pi,\uF]\ge0$ for all $s\in S_2$.
\end{enumerate}
\end{proposition}

\begin{proof}
(i)$\Rightarrow$(ii): Suppose a probability space $(\cX,\Sigma,\pi)$ and a random variable $\uF$ on $\cX$ with values in $K(S_1)^n$ satisfy $s[\pi,\uF]\ge0$ for all $s\in S_2$. Let $\mu$ be the pushforward of $\pi$ induced by $\uF$. 
Then for all $P\in\cX$ we have $(\mu,\uF(P))\in \bcK(S_1,S_2)\times K(S_1)$ and $f(\mu,\uF(P))=f[\pi,\uF](P)$. Thus (i) implies (ii).
Conversely, (i) is a special case of (ii) (where the probability space is $K(S_1)$ endowed with the $\sigma$-algebra of Borel sets and the measure $\mu$, and the coordinate functions are considered as random variables), so (ii) implies (i).
\end{proof}

We have introduced the evaluations of the second type mainly to indicate the scope of our framework, and they appear only in applications in Sections \ref{sss:cov} and \ref{sss:biloc};
in the rest of the paper, we deal with evaluations of the first type for the sake of simplicity (and this is sufficient by Proposition \ref{p:evals}). The following statement is a straightforward consequence of definitions, and represents the easy part of the algebra-geometry correspondence for moment polynomial positivity that is pursued in this paper. Namely, elements of the moment quadratic module are nonnegative on the corresponding moment semialgebraic set.

\begin{proposition}\label{p:ezi}
Let $S_1\subseteq\px$ and $S_2\subseteq\mp$. If $f\in\QM{S_1, S_2}$ then $f\ge0$ on $\bcK(S_1,S_2)\times K(S_1)$.
\end{proposition}

Let us comment on the restriction to probability measures in this paper. While moment polynomials can be evaluated on more general measures, the algebraic certificates of this paper involve constant term perturbations, or quotients of moment polynomials. The fact that such certificates indeed guarantee positivity is rooted in the integral of the constant 1 being a fixed positive number; or, from an algebraic perspective, that $\y\colon\MP\to\mp$ is $\mp$-linear. In general, our results can be modified to moment evaluations on measures $\mu$ with $\int\d \mu=t$ for a fixed $t\in\R_{>0}$, by simply renormalizing the measures to probability measures. Thus we restrict ourselves to probability measures for the sake of simplicity.

We conclude this section with a renowned quadrature result, which is also relevant for evaluations of moment polynomials, and is utilized in several subsequent proofs.

\begin{proposition}[Tchakaloff's theorem {\cite[Theorem 2]{tchakaloff}}]\label{p:tchakaloff}
Let $S_1\subseteq\px$. For every $\mu\in \prob{K(S_1)}$ and $d\in\N$ there exists $\nu\in \prob{K(S_1)}$ with $|\supp\nu|\le \binom{n+d}{d}$ such that $\y_{i_1,\dots,i_n}(\nu)=\y_{i_1,\dots,i_n}(\mu)$ for all $i_1+\cdots+i_n\le d$.

Thus if $S_2\subset \mp$ is finite, $f\in\MP$ and $d=\max\{\deg f, \deg s\colon s\in S_2\}$, then $f\ge0$ on $\bcK(S_1,S_2)\times K(S_1)$ if and only if $f\ge0$ on $\left\{\nu\in\bcK(S_1,S_2)\colon |\supp\nu| \le \binom{n+d}{d} \right\}\times K(S_1)$.
\end{proposition}

\begin{remark}\label{r:popbad}
Proposition \ref{p:tchakaloff} in principle allows to reformulate a moment polynomial optimization problem as a classical polynomial optimization problem in the following way. Let $S_1\subset\px$ and $S_2\subset\mp$ be finite, and $f\in\MP$. Denote $d=\max\{\deg f,\deg s\colon s\in S_2\}$ and $D=\binom{n+d}{d}$. By Proposition \ref{p:tchakaloff}, the infimum of $f$ on $\bcK(S_1,S_2)\times K(S_1)$ is equal to

\begin{equation}\label{e:mom2clas}
\begin{split}
&\inf_{\uX,\uY_1,\dots,\uY_D,\alpha_1,\dots,\alpha_D} \
f\left(\sum_{i=1}^D\alpha_i\delta_{\uY_i},\uX \right)\\
&\text{subject to }\ s\left(\sum_{i=1}^D\alpha_i\delta_{\uY_i},\uX \right)\ge0 \quad \text{for } s\in S_2,\\
&\phantom{subject to }\quad \uX,\uY_1,\dots,\uY_D\in K(S_1), \\
&\phantom{subject to }\quad \alpha_1,\dots,\alpha_D\ge0,\ \sum_{i=1}^D\alpha_i=1.
\end{split}
\end{equation}
Here, $\delta_{\uY_i}$ denotes the Dirac delta measure concentrated at $\uY_i\in\R^n$. While \eqref{e:mom2clas} minimizes a polynomial function subject to polynomial constraints, and can be thus approached with standard methods of polynomial optimization, it has $n+Dn+D=(\binom{n+d}{d}+1)(n+1)-1$ variables. 
This number quickly rises beyond the capabilities of solvers for global nonlinear optimization (see Subsection \ref{sss:cov} for a concrete example). 
Furthermore, the problem \eqref{e:mom2clas} does not fully utilize the structure of moment polynomials; e.g. $\uY_i$ can be permuted, $\alpha_i$ always appears jointly with $\uY_i$, and so on.
A moment polynomial optimization procedure bypassing these issues is developed in Section \ref{s:mompop} below.
\end{remark}

\subsection{Comparison with existing work}

Now that moment polynomials and their evaluations have been rigorously introduced, let us compare them to related constructions in the literature.

\subsubsection{State and trace polynomials}
Noncommutative analogs of the moment polynomial positivity framework have been introduced in \cite{KMV} and \cite{KMVW}. While moment polynomials are polynomial expressions in commuting variables $x_1,\dots,x_n$ and formal symbols of their products $\y(x_1^{i_1}\cdots x_n^{i_n})$, \emph{state polynomials} of \cite{KMVW} are polynomial expressions in \emph{noncommuting} variables $x_1,\dots,x_n$ and formal symbols of their products $\varsigma(x_{j_1}\cdots x_{j_\ell})$, called \emph{formal states}. While moment polynomials are evaluated on pairs of a point $\uX\in\R^n$ and a measure $\mu$ on $\R^n$, state polynomials are evaluated on pairs of a tuple of self-adjoint operators $\uX\in\cB(\cH)^n$ and a state $\lambda$ on $\cB(\cH)$ (a state is a unital positive linear functional $\cB(\cH)\to\C$).
If one is only interested in evaluations of state polynomials on tracial states (a state $\lambda:\cB(\cH)\to\C$ is tracial if $\lambda(AB)=\lambda(BA)$ for all $A,B\in\cB(\cH)$), then one can impose additional tracial relations on the symbol $\varsigma$, namely $\varsigma(w_1w_2)=\varsigma(w_2w_1)$ for all products $w_1,w_2$ of $x_1,\dots,x_n$, and thus obtains \emph{trace polynomials} \cite{KMV}, which originate in invariant theory \cite{P76}. 
There is hierarchy among moment, trace and state polynomial frameworks. First, positivity of trace polynomials subject to constraints $S$ can be viewed as positivity of state polynomials subject to $S$ together with tracial constraints, namely $S\cup \{\pm(\varsigma(w_1w_2)-\varsigma(w_2w_1))\colon w_1,w_2\}$. Second, positivity of moment polynomials subject to constraints $S$ can be viewed as positivity of trace polynomials subject to $S$ together with commutativity constraints, namely $S\cup \{\pm(x_ix_j-x_jx_i)\colon i,j\}$; this equivalence is exact because states on jointly commuting (bounded) operators correspond to (compactly supported) measures by the spectral theorem \cite[Theorem 5.23]{schmUnbded}.
All three frameworks are applied in quantum information theory. Classical shared randomness, maximal quantum entanglement, and general quantum entanglement can be investigated with moment, trace and state polynomials, respectively.

With the above hierarchy in mind, it is 
reasonable to ask why these three frameworks are investigated separately. The answer is that the scope and the strength of positivity certificates vary considerably among the frameworks. Let us demonstrate this with the main results presented in the introduction. When positivity without constraints is considered, the state analog of Hilbert's 17th problem admits a positive resolution, while trace and moment analogs of Hilbert's 17th problem do not. Nevertheless, Theorem \ref{ta:c} exactly identifies to what extent such a resolution fails for moment polynomials; namely, a moment polynomial admits a sum-of-squares certificate with denominators if and only if it is nonnegative under all pseudo-moment evaluations. Next, when positivity on bounded domains is considered, Theorem \ref{ta:a} provides a Positivstellensatz for arbitrary moment polynomials; on the other hand, within the other two frameworks, such a statement only holds for \emph{pure} variants of state and trace polynomials (polynomial expressions in states or traces, without freely appearing noncommuting variables $x_1,\dots,x_n$). Finally, Theorem \ref{ta:d} establishes a perturbative Positivstellensatz for moment polynomial positivity subject to arbitrary (not necessarily bounded) moment semialgebraic sets; on the contrary, no such general result exists for state or trace polynomials. The underlying principle of these differences is that the commutative theory (real algebraic geometry, measure theory) provides more sophisticated techniques than the noncommutative theory (free real algebraic geometry, operator algebras).

\subsubsection{Related advances on polynomial inequalities in moments}
Let us highlight three other works that can be compared to the moment polynomial framework of this paper.
In \cite{blekherman2022}, positivity of pure moment \emph{binomials} (i.e., polynomials with only two additive terms) is analyzed using tropicalization, resulting in a combinatorial description of nonnegative pure moment binomials subject to binomial inequalities.
The certificates of \cite{blekherman2022} are more definitive than the ones of this paper (easier to utilize, and of combinatorial nature as opposed to Theorem \ref{t:lass} for instance), 
but they only apply to binomials.
Next, \cite{fantuzzi} investigates polynomial inequalities in integrals of function variables and their derivatives. 
This setup can be compared with our framework on moment polynomial inequalities in random variables in light of Proposition \ref{p:evals}.
The setup in \cite{fantuzzi} is more general because it considers moments of not only function variables, but also of their derivatives. However, \cite{fantuzzi} derives only sufficiency conditions for polynomial integral inequalities, which are comparable to the reverse implication in Theorem \ref{ta:a}, but does not establish their necessity.
Lastly, let us refer to \cite{infusino2014,infusino2023moment} on infinite-dimensional moment problems. 
Often, moment problems are duals of sum-of-squares positivity certificates, and our paper provides the latter for moment polynomials, which are symmetric algebras on an infinite-dimensional space (spanned by the formal moments).
However, this point of view is not compatible with the scope of \cite{infusino2014,infusino2023moment}.
Namely, these papers investigate moment problems for algebras generated by nuclear spaces \cite{schaefer} (e.g., infinitely differentiable functions on $\R^n$ with compact support form a nuclear space), and the topological assumption of nuclearity plays a crucial role in their moment problem solvability criteria. On the other hand, our positivity certificates are based on algebraic manipulations of formal moment symbols without substantial topological considerations, and leverage the real-algebraic nature of $\px$, which leads to a computationally viable framework. Therefore our positivity certificates do not aspire to a moment problem, nor do moment problems associated to nuclear spaces carry implications for moment polynomial positivity.

\section{Pseudo-moments and Hilbert's 17th problem for moment polynomials}\label{sec:h17}

In this section we consider pseudo-moment evaluations of moment polynomials. We give a solution to a natural version of Hilbert's 17th problem for pseudo-moment evaluations (Theorem \ref{t:h17pseudo}). In particular, since positivity on pseudo-moments is stricter than positivity on moments, our solution implies that moment polynomials with nonnegative moment evaluations are not necessarily rational consequences of terms of the form $g^2$ and $\y(g^2)$ for $g\in\MP$.

Let $[\ux]_d$ denote all monomials in $\px$ of degree at most $d$, ordered degree-lexicographically according to $x_1>\cdots>x_n$. 
For $d\in\N$ let $H_d = (uv)_{u,v\in [\ux]_d}$ be the \emph{symbolic Hankel matrix} over $\px$ of order $d$. For any map $\alpha$ on $\px$, let $\alpha(H_d)$ denote the matrix obtained by applying $\alpha$ entry-wise to $H_d$. The relation between the matrix $H_d$ and moments of squares in $\MP$ is captured by the following statement.

\begin{lemma}\label{l:pseudo}
Let $\phi:\MP\to\R$ be a homomorphism of $\R$-algebras. The following statements are equivalent:
\begin{enumerate}[\rm (i)]
    \item $\phi(\y(f^2))\ge0$ for all $f\in\MP$;
    \item $\phi(\y(p^2))\ge0$ for all $p\in\px$;
    \item $(\phi\circ\y)(H_d)$ is positive semidefinite for all $d\in\N$.
\end{enumerate}
Such homomorphism is called a \emph{pseudo-moment evaluation}.
\end{lemma}

\begin{proof}
The implication (i)$\Rightarrow$(ii) is clear.
To see (ii)$\Rightarrow$(iii), suppose that (ii) holds, and let $w=(w_u)_{u\in [\ux]_d}$ be a real vector. Then
\begin{align*}
w^*\left((\phi\circ\y)(H_d)\right)w
=&\sum_{u,v\in[\ux]_d} w_uw_v\phi(\y(uv))
=\phi\left(\y\bigg(\sum_{u,v\in[\ux]_d} w_uw_v uv \bigg)\right)\\
=&\phi\left(\y\Bigg(\bigg(\sum_{u\in[\ux]_d} w_u u\bigg)^2 \Bigg)\right)\ge0
\end{align*}
by (ii). Since $w$ was arbitrary, it follows that $(\phi\circ\y)(H_d)$ is a positive semidefinite matrix, i.e., (iii) holds.
To see (iii)$\Rightarrow$(i), suppose that (iii) holds.
Let $f\in\MP$; for some $d\in\N$ we we can write it as $f=\sum_{u\in [\ux]_d} q_u u$ for $q_u\in\mp$. Then
$$
\phi\left(\y(f^2)\right)
=\sum_{u,v\in [\ux]_d} \phi(q_u)\phi(q_v)\phi\left(\y(uv)\right) 
=w^*\left((\phi\circ\y)(H_d)\right)w,
$$
where $w=(\phi(q_u))_{u\in [\ux]_d}$ is a real vector. Since $(\phi\circ\y)(H_d)$ is positive semidefinite, we have $\phi\left(\y(f^2)\right)\ge0$, so (i) holds.
\end{proof}

\begin{remark}\label{r:pseudoL}
A pseudo-moment evaluation $\phi:\MP\to\R$ is uniquely determined by $\phi(x_j)\in\R$ for $j=1,\dots,n$, and a unital linear functional $L:\px\to\R$ given by $L(p)=\phi(\y(p))$ and satisfying $L(p^2)\ge0$ for all $p\in\px$.
\end{remark}

\begin{remark}
There is a certain subtlety in Lemma \ref{l:pseudo}. Namely, the implication (ii)$\Rightarrow$(i) fails in general for homomorphisms $\MP\to R$ where $R$ is a closed real field containing $\R$, even when $n=1$. Indeed, by \cite[Example 2.6]{KPV} there exist a real closed field $R$ and a homomorphism $\phi:\mp\to R$ such that $\phi(\y(p^2))\ge0$ for all $p\in\px$, and $\phi(\y_2-\y_1^2)<0$, even though $\y_2-\y_1^2=\det\y(H_2)=\y((x_1-\y_1)^2)\in\qm{\emptyset,\emptyset}$.
This example stems from the fact that the quadratic module $\qm{\emptyset,\emptyset}$ in $\mp$, which is generated by $\y(f^2)$ for $f\in\MP$, cannot be generated by a finite set. Consequently, one cannot apply Tarski's tranfer principle (see e.g. \cite[Theorem 11.2.1]{marshallbook}) 
to translate the validity of (ii)$\Rightarrow$(i) for homomorphisms into $\R$ to homomorphisms into general real closed fields.
\end{remark}

Pseudo-moment evaluations form a strictly larger class than moment evaluations. For example, the pure moment polynomial $\y_{4,2}\y_{2,4}-\y_{2,2}^3$ is nonnegative under all moment evaluations, but not under all pseudo-moment evaluations.
This is shown in \cite[Example 5.1]{blekherman2022} (see Example \ref{ex:h17} below for an alternative argument) where, more generally, \emph{binomial} moment and pseudo-moment inequalities are characterized using tropicalization techniques.
Theorem \ref{t:h17pseudo} below gives a sums of squares certificate with denominators for moment polynomials that are nonnegative under all pseudo-moment evaluations.
For a sums of squares certificate with perturbations for moment polynomials that are nonnegative under all moment evaluations, see Theorem \ref{t:lass}.

In preparation for Theorem \ref{t:h17pseudo}, the next two lemmas address extensions and strictly positive approximations of positive functionals on $\R[\ux]_d$, which are relevant to pseudo-moment evaluations in light of Remark \ref{r:pseudoL}.

\begin{lemma}\label{l:hank_ext}
If a linear functional $L:\px_{2d}\to\R$ satisfies $L(p^2)>0$ for $p\in\px_d\setminus\{0\}$, then it extends to a linear functional $\widetilde L:\px\to\R$ satisfying $\widetilde L(p^2)>0$ for $p\in\px\setminus\{0\}$.
\end{lemma}

\begin{proof}
For $\alpha>0$ consider the linear functional
$L_\alpha: \px_{2d+2}\to\R$ defined on monomials $u\in [\ux]_{2d+2}$ as follows:
$$
L_\alpha(u)=\left\{
\begin{array}{ll}
L(u) & \text{ if }u\in\ [\ux]_{2d},\\
0 & \text{ if }u\in\ [\ux]_{2d+1}\setminus [\ux]_{2d},\\
\alpha\int_{[0,1]^n} u \d x_1\cdots\d x_n & 
\text{ if }u\in\ [\ux]_{2d+2}\setminus [\ux]_{2d+1}.\\
\end{array}
\right.
$$
Applying $L_\alpha$ entry-wise to $H_{d+1}$ results in
$$L_{\alpha}(H_{d+1})=
\begin{pmatrix}
L(H_d) & B^* \\ B & \alpha K
\end{pmatrix},
$$
where $L(H_d),K,B$ are independent of $\alpha$, and $L(H_d)$ and $K$ are positive definite matrices. Since $L(H_d)$ is invertible, $L(H_{d+1})$ is positive definite if and only if the Schur complement $\alpha K-BL(H_d)^{-1}B^*$ is positive definite. This is indeed the case for a sufficiently large $\alpha>0$. For such an $\alpha$, the functional $L_\alpha$  on $\px_{2d+2}$ is positive on $\y(p^2)$ for $p\in\px_{d+1}\setminus\{0\}$, and agrees with $L$ on $\px_{2d}$. Continuing in this fashion by induction on $d$, we obtain $\widetilde L:\px\to\R$ that extends $L$ and satisfies $\widetilde L(p^2)>0$ for $p\in\px\setminus\{0\}$.
\end{proof}

\begin{lemma}\label{l:hank_perturb}
For every $\ve>0$, $d\in\N$, and a unital linear functional $L:\px_{2d}\to\R$ satisfying $L(p^2)\ge0$ for $p\in\px_d$, there exists 
a unital linear functional $\widetilde L:\px_{2d}\to\R$ satisfying $\widetilde L(p^2)>0$ for $p\in\px_d\setminus\{0\}$, and $|\widetilde L(u)-L(u)|<\ve$ for $u\in\px_{2d}$.
\end{lemma}

\begin{proof}
Note that the unital linear functional $L_0:\px\to\R$ given by $$L_0(p)=\int_{[0,1]^n} p \d x_1\cdots\d x_n$$ satisfies $L_0(p^2)>0$ for $p\in\px\setminus\{0\}$. Then $\widetilde L = (1-\delta) L+\delta L_0$ for a sufficiently small $\delta>0$ has the desired properties.
\end{proof}

The proof of Theorem \ref{t:h17pseudo} requires some additional terminology from real algebra \cite{marshallbook}.
A \emph{preordering} $P$ in a commutative unital ring $A$ is a quadratic module closed under multiplication.
Let $\Omega$ denote the preordering in $\mp$ generated by $\qm{\emptyset,\emptyset}$, and let $\widehat\Omega$ denote the preordering in $\MP$ generated by $\QM{\emptyset,\emptyset}$. More concretely,
\begin{align*}
\Omega&=\left\{
\sum_i \y(f_{i1}^2)\cdots \y(f_{ik_i}^2)\colon f_{ij}\in \MP
\right\},\\
\widehat\Omega&=\left\{
\sum_i f_{i0}^2\y(f_{i1}^2)\cdots \y(f_{ik_i}^2)\colon f_{ij}\in \MP
\right\}.\\
\end{align*}
Preorderings play a crucial role in the Krivine-Stengle Positivstellensatz \cite[Theorem 2.2.1]{marshallbook}, which we recall now. If $f,g_1,\dots,g_\ell\in\R[y_1,\dots,y_k]$ and $f$ is nonnegative on $\{g_1\ge0,\dots,g_\ell\ge0\}$, then the Krivine-Stengle Positivstellensatz implies that $f$ is a quotient of elements from the preordering in $\R[y_1,\dots,y_k]$ generated by $g_1,\dots,g_\ell$.

The following lemma relates the symbolic Hankel matrix $H_d$ to the preordering $\Omega$.

\begin{lemma}\label{l:minors}
Every principal minor of $\y(H_d)$ is a quotient of elements in $\Omega$.
\end{lemma}

\begin{proof}
This is an adaptation of \cite[Proposition 4.2]{KMVW}; let us sketch the main idea without technical details. By \cite[Proposition 4.3]{klep2018positive}, a symmetric matrix $X$ of a fixed size $D$ is positive semidefinite if and only if $\tr(X T_1(X)^2)\ge0,\dots,\tr(X T_\ell(X)^2)\ge0$, where $T_j$ are polynomials in $X$ and traces of its powers, dependent only on $D$. In particular, by the Krivine-Stengle Positivstellensatz, minors of $X$ are quotients of elements from the preordering in $\R[\tr(X),\dots,\tr(X^D)]$ generated by $\tr(X T_1(X)^2),\dots,\tr(X T_\ell(X)^2)$. This fact is then applied to $X= \y(H_d)$, in which case one can furthermore verify that $\tr(\y(H_d) T_j(\y(H_d))^2)\in\qm{\emptyset,\emptyset}$ for $j=1,\dots,\ell$.
\end{proof}

\begin{theorem}\label{t:h17pseudo}
Let $f\in\MP$. Then all pseudo-moment evaluations of $f$ are nonnegative if and only if $f$ is a quotient of elements in $\widehat\Omega$.
\end{theorem}

\begin{proof}
$(\Rightarrow)$ Let $d=\deg f$. Let $A$ be the polynomial ring generated by $x_1,\dots,x_n$ and $\y(u)$ for $u\in [\ux]_{2d}$; in particular, $A$ is a finitely generated polynomial ring. Let $P$ be the preordering in $A$ generated by the principal minors of $\y(H_d)$. Then $P$ is a finitely generated preordering, and contained in $\widehat\Omega$ by Lemma \ref{l:minors}.
Assume $f$ is not a quotient of elements in $\widehat\Omega$. Then $f$ is not a quotient of elements in $P$. By the Krivine-Stengle Positivstellensatz \cite[Theorem 2.2.1]{marshallbook} there is a homomorphism $\phi:A\to\R$ such that $\phi(f)<0$ and $\phi$ is nonnegative on the principal minors of $\y(H_d)$. In particular, the matrix $(\phi\circ\y)(H_d)$ is positive semidefinite. Note that $\phi$ is determined by $\phi(x_j)$ for $j=1,\dots,n$ and the linear functional $L:\px_d\to\R$ given by $L(p)=\phi(\y(p))$. By Lemma \ref{l:hank_perturb}, we can slightly perturb $L$, so that $L(p^2)>0$ for $p\in\px_d\setminus\{0\}$, and still $\phi(f)<0$. By Lemma \ref{l:hank_ext}, $L$ extends to $\widetilde L:\px\to\R$ such that $\widetilde L(p^2)\ge0$ for $p\in\px$. Define a homomorphism $\widetilde\phi:\MP\to\R$ determined by $\widetilde\phi(p) = \phi(p)$ and $\widetilde\phi(\y(p)) = \widetilde L(p)$ for $p\in\px$. Then $\widetilde\phi$ is a pseudo-moment evaluation by Lemma \ref{l:pseudo}, and $\widetilde\phi(f)<0$.

$(\Leftarrow)$ Let $f =\frac{g}{h}$ for some nonzero $g,h\in\widehat\Omega$. Suppose $\phi(f)<0$ for a pseudo-moment evaluation $\phi:\MP\to\R$. Note that $\psi(g),\psi(h)\ge0$ for all pseudo-moment evaluations $\psi:\MP\to\R$. Since $h\neq0$, there exists 
a pseudo-moment evaluation $\phi':\MP\to\R$ such that $\phi'(h)>0$ (indeed, one can even choose $\phi'$ arising from evaluation at a point in $\R^n$ and a measure in $\prob{\R^n})$. For $\ve\in[0,1]$ let $\phi_\ve:\MP\to\R$ be the pseudo-moment evaluation determined by $\phi_\ve(\y(p)) = (1-\ve)\phi(\y(p))+\ve \phi'(\y(p))$ for $p\in\px$. Then there exists a sufficiently small $\ve>0$ so that $\phi_\ve(f)<0$ and $\phi_\ve(h)>0$. Hence
$$0>\phi_\ve(h)\phi_\ve(f)=\phi_\ve(g)\ge0,$$
a contradiction.
\end{proof}

\begin{example}[{\cite[Example 5.1]{blekherman2022}}]\label{ex:h17}
Let $f=\y_{4,2}\y_{2,4}-\y_{2,2}^3$. All moment evaluations of $f$ are nonnegative by H\"older's inequality applied to the three polynomials $x_1^4x_2^2,x_1^4x_2^2,1$ as
$$\int (x_1^4x_2^2\cdot x_1^4x_2^2\cdot 1)\d\mu\le
\left(\int x_1^4x_2^2\d\mu\right)^{\frac13}
\left(\int x_1^4x_2^2\d\mu\right)^{\frac13}
\left(\int 1\d\mu\right)^{\frac13}
.$$
On the other hand, consider the functional $L:\px_6\to\R$ given on the Hankel matrix $H_3$
as
\begin{equation}\label{e:greg}
L(H_3) = \begin{pmatrix}
 1 & 0 & 0 & 5 & 0 & 5 & 0 & 0 & 0 & 0 \\
 0 & 5 & 0 & 0 & 0 & 0 & 26 & 0 & 2 & 0 \\
 0 & 0 & 5 & 0 & 0 & 0 & 0 & 2 & 0 & 563 \\
 5 & 0 & 0 & 26 & 0 & 2 & 0 & 0 & 0 & 0 \\
 0 & 0 & 0 & 0 & 2 & 0 & 0 & 0 & 0 & 0 \\
 5 & 0 & 0 & 2 & 0 & 563 & 0 & 0 & 0 & 0 \\
 0 & 26 & 0 & 0 & 0 & 0 & 587 & 0 & 1 & 0 \\
 0 & 0 & 2 & 0 & 0 & 0 & 0 & 1 & 0 & 1 \\
 0 & 2 & 0 & 0 & 0 & 0 & 1 & 0 & 1 & 0 \\
 0 & 0 & 563 & 0 & 0 & 0 & 0 & 1 & 0 & 319642
\end{pmatrix}
\end{equation}
Note that the right-hand side of \eqref{e:greg} is positive definite, and $L(x_1^4x_2^2)L(x_1^2x_2^4)-L(x_1^2x_2^2)^3 = 1-2^3=-7$. By Lemma \ref{l:hank_ext}, $L$ extends to $\widetilde L:\px\to\R$ such that $\widetilde L(p^2)\ge0$ for $p\in\px$.
Therefore, $\phi(f)<0$ for the pseudo-moment evaluation $\phi$ determined by $\widetilde L$ (and any evaluation on $x_1,x_2$),
so $f$ is neither a quotient of sums of products of elements in $\QM{\emptyset,\emptyset}$, nor in $\qm{\emptyset,\emptyset}$, by Theorem \ref{t:h17pseudo}.
\end{example}

\section{Archimedean Positivstellensatz for moment polynomials}
\label{sec:arch}

The main result of this section, Theorem \ref{t:arch}, describes moment polynomials that are positive subject to constraints on measures with a given compact support.
Recall \cite[Section 5.2]{marshallbook} that a quadratic module $A$ in a commutative unital ring $A$ is \emph{archimedean} if for every $a\in A$ there exists $N\in\N$ such that $N\pm a\in M$.
Equivalently, a quadratic module $M$ in $\px$ is archimedean if and only if there is an $N\in\N$ such that $N-x_1^2-\cdots-x_n^2\in M$ \cite[Corollary 5.2.4]{marshallbook}.
We start by observing how archimedianity in $\px$ transfers to archimedianity in $\MP$.

\begin{lemma}\label{l:arch}
Let $S\subseteq\px$.
If $\QM{S}\subseteq\px$ is archimedean, then $\qm{S,\emptyset}\subseteq\mp$ and $\QM{S,\emptyset}\subseteq\MP$ are archimedean.
\end{lemma}

\begin{proof}
Let $(i_1,\dots,i_n)\in\N_0^n$ be arbitrary. Since $\QM{S}$ is archimedean, there exists $N>0$ such that $N\pm x_1^{i_1}\cdots x_n^{i_n}$ is a convex combination of some $p^2s$ for $p\in\px$ and $s\in S\cup\{1\}$. Hence, $N\pm \y_{i_1,\dots,i_n} \in \qm{S,\emptyset}$. 
Consequently, $\qm{S,\emptyset}$ and $\QM{S,\emptyset}$ are archimedean by \cite[Proposition 5.2.3]{marshallbook}.
\end{proof}

The following theorem is the main result of this section.

\begin{theorem}[Archimedean Positivstellensatz]\label{t:arch}
Let $S_1\subseteq\px$ and $S_2\subseteq\mp$, and suppose $\QM{S_1}$ is archimedean in $\px$. 
The following statements are equivalent for $f\in\MP$:
\begin{enumerate}[\rm (i)]
    \item $f\ge0$ on $\bcK(S_1,S_2)\times K(S_1)$;
    \item $f+\ve\in \QM{S_1, S_2}$ for all $\ve>0$.
\end{enumerate}
The following statements are equivalent for $f\in\mp$:
\begin{enumerate}[\rm (i')]
	\item $f\ge0$ on $\bcK(S_1,S_2)$;
	\item $f+\ve\in \qm{S_1, S_2}$ for all $\ve>0$.
\end{enumerate}
\end{theorem}

\begin{proof}
We only prove the first equivalence (the proof of second one is analogous). The implication (ii)$\Rightarrow$(i) is straightforward. 
Now suppose (ii) is false. 
Since $\QM{S_1,S_2}$ is archimedean, the Kadison-Dubois representation theorem 
\cite[Theorem 5.4.4]{marshallbook} states that $\QM{S_1,S_2}$ contains all $g\in\MP$ which are positive under all homomorphisms $\MP\to\R$ that are nonnegative on $\QM{S_1,S_2}$. In particular, since $f+\ve\notin\QM{S_1,S_2}$, there exists a homomorphism $\varphi:\MP\to\R$ such that $\varphi(\QM{S_1,S_2})=\R_{\ge0}$ and $\varphi(f+\ve)\le0$ (in particular, $\varphi(f)<0$). Then $\uX:=(\varphi(x_1),\dots,\varphi(x_n))\in K(S_1)$. Consider the unital functional $L:\px\to\R$ given by $L(p)=\varphi(\y(p))$. Then $L$ is nonnegative on $\QM{S_1}$, so by the solution of the moment problem on compact sets \cite[Theorem 12.36 (ii)]{schmbook} there is $\mu\in \prob{K(S_1)}$ such that $L(p)=\int p\d\mu$ for all $p\in\px$. By the construction, $\mu\in \bcK(S_1,S_2)$. Therefore, $(\mu,\uX)\in \bcK(S_1,S_2)\times K(S_1)$ and $f(\mu,\uX)=\varphi(f)<0$.
\end{proof}

\begin{remark}\label{r:ncstate}
The equivalence (i')$\Leftrightarrow$(ii') in Theorem \ref{t:arch} also follows from \cite[Theorem 5.5]{KMVW} on state polynomials and their evaluations on constrained tuples of bounded operators and states. Indeed, the class of admissible constraints in \cite[Theorem 5.5]{KMVW} is large enough to allow for commutators, and thus one can consider positivity of state polynomials on commuting bounded operators subject to archimedean constraints. The second part of Theorem \ref{t:arch} can be then obtained using the spectral theorem for tuples of commuting bounded operators \cite[Theorem 5.23]{schmUnbded}.
However, note that the results of \cite{KMVW} carry implications only for pure moment polynomials, but not for general moment polynomials, and are not applicable to the equivalence (i)$\Leftrightarrow$(ii) in Theorem \ref{t:arch}.
\end{remark}

\begin{corollary}\label{c:strict}
Let $S_1\subset\px$ and $S_2\subset\mp$, and suppose $\QM{S_1}$ is archimedean in $\px$. If $f\in\MP$ is strictly positive on $\bcK(S_1,S_2)\times K(S_1)$, then $f\in \QM{S_1,S_2}$.
\end{corollary}

\begin{proof}
Since $K(S_1)$ is compact, the set of Borel probability measures supported on $K(S_1)$ is also compact by \cite[Theorem II.6.4]{parthasarathy2005}, and is equal to $\prob{K(S_1)}$ (the existence of all marginal moments for Borel measures on a compact subset of $\R^n$ is automatic). Therefore, $\bcK(S_1,S_2)\times K(S_1)$ is compact, so there is $\ve>0$ such that $f-\ve\ge0$ on $\bcK(S_1,S_2)\times K(S_1)$. Then $f\in\QM{S_1,S_2}$ by Theorem \ref{t:arch}.
\end{proof}

Corollary \ref{c:strict} also admits the following interpretation in the absence of the archimedean assumption.

\begin{corollary}\label{c:bded}
Let $S_1\subset\px$ and $S_2\subset\mp$, and suppose $K(S_1)\subset\R^n$ is bounded. Then the following statements are equivalent for $f\in\MP$:
\begin{enumerate}[\rm (i)]
    \item $f\ge0$ on $\bcK(S_1,S_2)\times K(S_1)$;
    \item $f+\ve\in \QM{\widetilde{S_1},S_2}$ for all $\ve>0$, where $\widetilde{S_1}$ is the set of all square-free products of elements in $S_1$.
\end{enumerate}
\end{corollary}

\begin{proof}
If $K(S_1)$ is bounded, then $\QM{\widetilde{S_1}}$ is archimedean in $\px$ by \cite[Corollary 6.1.2]{marshallbook}. The rest then follows from Theorem \ref{t:arch}.
\end{proof}

\section{Moment polynomial optimization and examples}\label{s:mompop}

Theorem \ref{t:arch} can be applied to design a converging hierarchy of semidefinite programs (SDPs) for moment polynomial optimization as presented in this section. For the sake of simplicity, we first focus on pure moment polynomial objective functions, and then indicate the necessary changes for general moment polynomial objective functions. Finally, we demonstrate how this SDP hierarchy can be applied to problems in quantum information theory (Subsections \ref{sss:cov} and \ref{sss:biloc}).

Let $S_1\subset\px$ and $S_2\subset\mp$ be finite, and $r\in\N$. Recall that $\qm{S_1,S_2}_{2r}$ is the convex hull of 
\begin{align*}
\y(f^2s_1),\, q^2s_2\colon\quad
&s_i\in \{1\}\cup S_i,\, f\in\MP,\, q\in\mp,\\
&\deg s_1+2\deg f,\,
\deg s_2+2\deg q\le 2r.
\end{align*}
Membership in $\qm{S_1,S_2}_{2r}$ can be certified by an SDP; 
indeed, its members are of the form
$$
\sum_{s\in \{1\}\cup S_1}
\sum_{v_1,v_2}
G^{(s)}_{v_1,v_2} \cdot
\y(v_1v_2 s)
+
\sum_{t\in S_2}
\sum_{u_1,u_2}
H^{(t)}_{u_1,u_2} \cdot u_1u_2 t,
$$
where $v_i$ are monomials in $\MP_{r-\frac{\deg s}{2}}$, $u_i$ are monomials in $\mp_{r-\frac{\deg t}{2}}$, 
and $G^{(s)}$, $H^{(t)}$ are positive semidefinite matrices of dimensions $\dim\MP_{r-\frac{\deg s}{2}}$ and $\dim\mp_{r-\frac{\deg t}{2}}$, respectively.

For $f\in\mp$ and $r\ge \frac{\deg f}{2}$ consider the sequence of SDPs
\begin{equation}\label{e:fstSDP}
f_r=\sup\{\alpha\in\R\colon f-\alpha \in \qm{S_1,S_2}_{2r}\}. 
\end{equation}
Theorem \ref{t:arch} implies (under the archimedean assumption) that $f_r$ form a sequence of upper bounds converging to the infimum of $f$ on $\bcK(S_1,S_2)$.

\begin{corollary}\label{c:hier}
Let $S_1\subseteq \px$, $S_2\subseteq\mp$, $f\in\mp$, and suppose $\QM{S_1}$ is archimedean in $\px$. Then the sequence 
$\{f_r\}_{r\ge \frac{\deg f}{2}}$ arising from the
SDP hierarchy \eqref{e:fstSDP} converges monotonically to $f_*:=\inf_{\mu\in \bcK(S_1,S_2)} f(\mu)$ from below.
\end{corollary}

\begin{proof}
The sequence $\{f_r\}_r$ is increasing since $\qm{S_1,S_2}_{2r}\subseteq \qm{S_1,S_2}_{2(r+1)}$. Also, $f_r\le f_*$ because $f-\alpha\in\qm{S_1,S_2}_{2r}$ implies $f\ge\alpha$ on $\bcK(S_1,S_2)$. Now let $\ve>0$ be arbitrary.
Since $f-f_*\ge0$ on $\bcK(S_1,S_2)$, Theorem \ref{t:arch} implies $f-f_*+\ve\in\qm{S_1,S_2}$. Thus, there exists $r\in\N$ such that
$$f-f_*+\ve \in \qm{S_1,S_2}_{2r},$$
so $f_*-\ve \le f_r$.
Therefore, $\lim_{r\to \infty}f_r=f_*$.
\end{proof}

Corollary \ref{c:hier} is also a specialization of \cite[Corollary 6.1]{KMVW} from the state polynomial setup. Let us note a few further consequences of \cite[Section 6]{KMVW} without proofs:
\begin{enumerate}
    \item If $N-x_1^2-\cdots-x_n^2$ for some $N>0$ is a conic combination of $S_1\cup\{\ell^2\colon \ell \in\px_1\}$, then there is no duality gap between SDP \eqref{e:fstSDP} and its dual,
$$\inf\left\{
L(f)\colon L\in \mp_{2r}^\vee,\,L(1)=1,\,
L(\qm{S_1,S_2}_{2r})=\R_{\ge0}
\right\}.$$
    \item If the solution of the dual of \eqref{e:fstSDP} satisfies certain rank conditions, then the SDP hierarchy \eqref{e:fstSDP} stops, and one can extract a concrete finitely supported optimizer for $f_*$.
    \item While the sizes of SDPs \eqref{e:fstSDP} and their duals grow quickly in concrete applications, one can mitigate this by employing sparsity \cite{sparsebook} and symmetry reductions.
\end{enumerate}

To apply semidefinite programming to optimization of general moment polynomials, one needs to first address the following obstacle. Recall that  $\QM{S_1,S_2}_{2r}$ is the convex hull of 
\begin{align*}
f_1^2\y(g^2s),\, f_2^2t\colon\quad
&s\in \{1\}\cup S_1,\, t\in S_1\cup S_2,\,f_1,f_2,g\in\MP,\\
&\deg s_1+2(\deg f_1+\deg g),
\deg s_2+2\deg f_2\le 2r.
\end{align*}
While membership in $\QM{S_1,S_2}_{2r}$ is a feasibility linear conic program \cite[Section IV.6]{Bar}, it does not seem to be a semidefinite program. For this reason we consider a slightly larger cone.
Sums of elements of the form $f^2s$, where $f\in\MP$ and $s\in S_1\cup S_2$ satisfy $\deg s+2\deg f\le 2r$, are precisely
\begin{equation}\label{e:tosdp1}
\sum_{s\in S_1\cup S_2}\sum_{v_1,v_2} G_{v_1,v_2}^{(s)}\cdot v_1v_2s
\end{equation}
where $v_i$ are monomials in $\MP_{r-\frac{\deg s}{2}}$, and $G^{(s)}$ are positive semidefinite matrices of dimensions $\dim \MP_{r-\frac{\deg s}{2}}$.
On the other hand, sums of elements of the form $f^2\y(g^2s)$, where $f,g\in\MP$ and $s\in \{1\}\cup S_1$ satisfy $\deg s+2(\deg f+\deg g)\le 2r$, can be written as
\begin{equation}\label{e:tosdp2}
\sum_{s\in\{1\}\cup S_1}\sum_{(u_1,v_1),(u_2,v_2)}
G_{(u_1,v_1),(u_2,v_2)}^{(s)} \cdot
\y(u_1u_2 s)v_1v_2,
\end{equation}
where $(u_i,v_i)$ are pairs of monomials $u_i$ in $\px$ and $v_i$ in $\MP$ with $\deg s+2(\deg u_i+\deg v_i)\le 2r$, and $G^{(s)}$ are positive semidefinite matrices of size $\sum_{i+j\le r-\frac{\deg s}{2}} (\dim\px_i+\dim\MP_j)$.
Let $\QQM{S_1,S_2}_{2r}$ denote the set of sums of elements of the form \eqref{e:tosdp1} and \eqref{e:tosdp2}.
Then the cone $\QQM{S_1,S_2}_{2r}$ contains $\QM{S_1,S_2}_{2r}$, but is typically larger than $\QM{S_1,S_2}_{2r}$. For example, if $n\ge 4$ then
$$
\y(x_1^2)x_2^2+2\y(x_1x_3)x_2x_4+\y(x_3^2)x_4^2
$$
belongs to $\QQM{S_1,S_2}_{2r}$ (namely, it is of the form \eqref{e:tosdp2}) but not to $\QM{S_1,S_2}_{2r}$.
Nevertheless, elements of $\QQM{S_1,S_2}_{2r}$ are nonnegative on $\bcK(S_1,S_2)\times K(S_1)$, and membership in $\QQM{S_1,S_2}_{2r}$ can be determined by a feasibility SDP.
Given $f\in \MP$, the optimization problems \begin{equation}\label{e:fstSDP2}
f_r=\sup\{\alpha\colon f-\alpha\in \QQM{S_1,S_2}_{2r}\}
\end{equation}
for $r\ge \frac{\deg f}{2}$ are then SDPs, and the following analog of Corollary \ref{c:hier} holds.

\begin{corollary}\label{c:hier2}
Let $S_1\subseteq \px$, $S_2\subseteq\mp$, $f\in\MP$, and suppose $\QM{S_1}$ is archimedean in $\px$. Then the sequence 
$\{f_r\}_{r\ge \frac{\deg f}{2}}$ arising from the
SDP hierarchy \eqref{e:fstSDP2} converges monotonically to $f_*:=\inf_{(\mu,\uX)\in \bcK(S_1,S_2)\times K(S_1)} f(\mu,\uX)$ from below.
\end{corollary}

Next, we demonstrate the above SDP hierarchy method on two open optimization problems arising from nonlinear Bell inequalities in quantum physics \cite{nobel}. 
A Bell scenario refers to a repeated experiment, where two or more parties, not communicating between themselves, perform measurements on a source of classical randomness and record outcomes in each round. Each party has a fixed number of measurements at disposal (with a fixed set of possible outputs), and uses one of them in each round. These measurements are viewed as random variables on a probability space that governs the source measured by the parties.
The conditional expectations of joint outcomes relative to measurements in such an experiment are called correlations. A Bell inequality for a given Bell scenario is an inequality in correlations that is valid regardless of the underlying probability space and measurements that model the experiment. Bell inequalities thus provide universal constraints on classical probabilistic models of a Bell scenario, and are widely used in quantum physics to certify quantum nonlocality (which violates Bell inequalities). 
Bell scenarios that model more sophisticated quantum networks call for nonlinear polynomial Bell inequalities, which are (from the perspective of this paper) special cases of moment polynomial inequalities.
Section \ref{sss:cov} confirms a covariance Bell inequality proposed in \cite{PHBB}, and Section \ref{sss:biloc} rectifies a bilocal Bell inequality proposed in \cite{TGB21}.

Encoding of moment polynomials and preparation of the SDPs was done in Mathematica, and numerical solutions of the SDPs were obtained with the SDP solver MOSEK. 
The final exact certificates provided below are heavily inspired by these numerical solutions. The ad-hoc Mathematica code for constructing and solving SDPs is available on GitHub: \href{https://github.com/magronv/mompop/}{\tt https://github.com/magronv/mompop}
(it terminates within a few minutes on a standard desktop PC), as well as the notebook for verifying the obtained exact certificates. 

\subsection{Covariance Bell inequality}\label{sss:cov}

\def\cov{\operatorname{cov}}
For $j=1,2,3$ let $A_j,B_j$ be binary random variables (valued in $\{-1,1\}$) on a probability space $(\cX,\Sigma,\pi)$, and consider the expression
\begin{align*}
\cov_{3322}(\underline{A},\underline{B}):=
&\cov(A_1,B_1)+\cov(A_1,B_2)+\cov(A_1,B_3) \\
+&\cov(A_2,B_1)+\cov(A_2,B_2)-\cov(A_2,B_3) \\
+&\cov(A_3,B_1)-\cov(A_3,B_2)
\end{align*}
where $\cov(X,Y)=\int XY\d\pi-\int X\d\pi\cdot\int Y\d\pi$.
In \cite{PHBB}, the authors ask for the largest possible value of $\cov_{3322}$.
They provide concrete examples of probability spaces (on a three-element set) and binary random variables where $\cov_{3322}$ attains the value $\frac92$.
In the quest for proving that $\cov_{3322}\le \frac92$ for all binary random variables, they propose a reduction to solving a certain number of linear systems. Nonetheless, for establishing this particular inequality, they estimate that more than $10^{14}$ linear systems would have to be solved, thus rendering this particular approach infeasible.
As an alternative, they suggest maximizing $\cov_{3322}$ via classical polynomial optimization similarly as in Remark \ref{r:popbad}. However, the corresponding polynomial problem has too many variables for global optimization tools to work. Thus they use numerical nonlinear optimization to look for local maxima of $\cov_{3322}$ from numerous starting points, which lends confidence to their conjecture that $\cov_{3322}\le 4.5$. Below, we settle this conjecture using the methods developed in this paper.

Let
\begin{align*}
f=
&\y_{100100}-\y_{100000}\,\y_{000100}
+\y_{100010}-\y_{100000}\,\y_{000010}
+\y_{100001}-\y_{100000}\,\y_{000001}\\
+&\y_{010100}-\y_{010000}\,\y_{000100}
+\y_{010010}-\y_{010000}\,\y_{000010}
-\y_{010001}+\y_{010000}\,\y_{000001}\\
+&\y_{001100}-\y_{001000}\,\y_{000100}
-\y_{001010}+\y_{001000}\,\y_{000010}.
\end{align*}
The question of \cite{PHBB} is equivalent to the moment polynomial optimization problem
$$f_* = \sup \ f \quad \text{subject to } x_j^2=1 \text{ for }j=1,\dots,6.$$
By Corollary \ref{c:hier}, we have $f_r\searrow f_*$ for
$$f_r=\inf\{\alpha\colon \alpha-f \in \qm{S,\emptyset}\}$$
with $S=\{\pm(1-x_j^2)\colon j=1,\dots,6\}$.
When constructing SDPs for $f_r$, we encode the relations of $S$ as substitution rules, to reduce the size of the SDPs. For $r=2$, the resulting SDP has 4146 indeterminates and the semidefinite constraint of size $100\times 100$, and yields $f_2=4.5$. 
Therefore, $f_* = 4.5$.
After carefully inspecting the output of the SDP solver (namely, heuristically extracting exact constraints from the kernel of the numerical positive semidefinite constraint) one can obtain an exact certificate for $f_*\le \frac92$. Concretely, modulo the relations $x_j^2=1$ we have
$\frac92-f=\y(v^*Gv)$, where
$$G=
\left(\begin{array}{rrrrrrrr}
 \frac{4}{3} & 0 & -\frac{1}{2} & 0 & 0 & 0 & 0 & 0 \\[1mm]
 0 & \frac{1}{32} & 0 & 0 & 0 & 0 & 0 & 0 \\[1mm]
 -\frac{1}{2} & 0 & \frac{3}{8} & \frac{1}{16} & 0 & 0 & 0 & 0 \\[1mm]
 0 & 0 & \frac{1}{16} & \frac{1}{8} & -\frac{1}{8} & 0 & \frac{3}{64} & \frac{3}{64} \\[1mm]
 0 & 0 & 0 & -\frac{1}{8} & \frac{1}{4} & -\frac{1}{8} & -\frac{1}{16} & -\frac{1}{16} \\[1mm]
 0 & 0 & 0 & 0 & -\frac{1}{8} & \frac{1}{4} & 0 & 0 \\[1mm]
 0 & 0 & 0 & \frac{3}{64} & -\frac{1}{16} & 0 & \frac{3}{64} & 0 \\[1mm]
 0 & 0 & 0 & \frac{3}{64} & -\frac{1}{16} & 0 & 0 & \frac{3}{64}
\end{array}\right)\in\R^{8\times 8}$$
is positive definite, and
$$v=\left(\begin{array}{c}
x_6\y_{110000}\\[1mm]
(x_1-x_2)(x_4+x_5)\\[1mm]
1+x_6(x_2-x_1+4\y_{100000})\\[1mm]
2-(x_1+x_2)(x_4+x_5)-8x_6\y_{100000}\\[1mm]
(x_3+2\y_{001000})(x_4-x_5)\\[1mm]
2x_4\y_{001000}-x_3x_5\\[1mm]
2+(x_4-x_5)(4\y_{001000}-x_1-x_2)+\frac83 x_4\y_{110000}+8x_6\y_{100000}\\[1mm]
2+(x_4-x_5)(4\y_{001000}+x_1+x_2)+\frac83 x_5\y_{110000}+8x_6\y_{100000}
\end{array}\right)\in\MP^8.
$$

\subsection{Bilocal Bell inequality}\label{sss:biloc}

In \cite{TGB21,tavakoli22}, the authors ask about the largest value of
\begin{equation}\label{e:tgb}
\frac13 \sum_{i\in\{1,2,3\} } \Big(\bE(B_iC_i)-\bE(A_iB_i)\Big)-\sum_{\{i,j,k\}=\{1,2,3\} } \bE(A_iB_jC_k)
\end{equation}
where $A_j,B_j,C_j$ for $j=1,2,3$ are binary random variables on a probability space $(\cX,\Sigma,\pi)$ satisfying bilocality constraints
\begin{equation}\label{e:biloc}
\bE(A_1^{k_1}A_2^{k_2}A_3^{k_3}C_1^{k_4}C_2^{k_5}C_3^{k_6})=
\bE(A_1^{k_1}A_2^{k_2}A_3^{k_3})\bE(C_1^{k_4}C_2^{k_5}C_3^{k_6})
\end{equation}
for all $k_i\in \{0,1\}$, and additional vanishing constraints
\begin{equation}\label{e:Z=0}
\begin{split}
&\bE(A_i)=\bE(B_i)=\bE(C_i)=0 \quad \text{for }
i\in \{1,2,3\}, \\
&\bE(A_iB_j)=\bE(B_iC_j)=0\quad \text{for }i\neq j, \\
&\bE(A_iB_jC_k)=0\quad \text{for }|\{i,j,k\}|\le2.
\end{split}
\end{equation}
Here, $\bE(X)=\int X\d \pi$.
In \cite{TGB21} it is shown that the largest value of \eqref{e:tgb} for bilocal models with the tetrahedral symmetry is 3. Furthermore, \cite{TGB21,tavakoli22} suggest that \eqref{e:tgb} can be at most 3 in general, and support this claim with numerical methods that search for local maxima.
However, as shown below, this claim is false; the largest value of \eqref{e:tgb} subject to \eqref{e:biloc} and \eqref{e:Z=0} is 4.

Consider the moment polynomial optimization problem
\begin{equation}\label{e:tgb_opt}
\begin{split}
&\sup \
\frac13 \sum_{i\in\{1,2,3\} } \Big(\y(x_{i+3}x_{i+6})-\y(x_ix_{i+3})\Big)-\sum_{\{i,j,k\}=\{1,2,3\} } \y(x_ix_{j+3}x_{k+6})\\
&\text{subject to }\\
& \y(x_1^{k_1}x_2^{k_2}x_3^{k_3}x_7^{k_4}x_8^{k_5}x_9^{k_6})=
\y(x_1^{k_1}x_2^{k_2}x_3^{k_3})\y(x_7^{k_4}x_8^{k_5}x_9^{k_6}) \quad \text{for } k_i\in\{0,1\}, \\
&x_j^2=1 \text{ and }\y(x_j)=0 \quad \text{for }
j\in \{0,\dots,9\}, \\
&\y(x_ix_{j+3})=\y(x_{i+3}x_{j+6})=0\quad \text{for } i,j\in\{1,2,3\},\ i\neq j, \\
&\y(x_ix_{j+3}x_{k+6})=0\quad \text{for }i,j,k\in\{1,2,3\},\ |\{i,j,k\}|\le2.
\end{split}
\end{equation}
Corollary \ref{c:hier} provides a converging sequence of upper bounds for the solution of \eqref{e:tgb_opt}. For $r=3$, one obtains the upper bound 4.0 by solving an SDP with 31017 indeterminates and the semidefinite constraint of size 263, or more practically, by solving its dual with 4549 indeterminates and the semidefinite constraint of size 325. The value of \eqref{e:tgb} subject to \eqref{e:biloc} and \eqref{e:Z=0} is thus at most 4.0.
Next, we show that the value 4 is attained; incidentally, the below construction was inspired by the numerical output of the dual SDP.
Denote
$$\eta_0=(\phantom{.}1\ \phantom{.}1\ \phantom{.}1\ \phantom{.}1),
\quad \eta_1=(\phantom{.}1\ \phantom{.}1\ -1\ -1),
\quad \eta_2=(\phantom{.}1\ -1\ \phantom{.}1\ -1),
\quad \eta_3=(\phantom{.}1\ -1\ -1\ \phantom{.}1),
$$
and let $e_i\in \R^4$ be the $i$\textsuperscript{th} standard unit vector.
Endow $\{1,2,3,4\}^2$ with the uniform probability distribution, and consider the following binary random variables on it:
$$
A_i=\eta_0\otimes \eta_i,\qquad
B_i=\left(\eta_0\otimes \eta_0-2\sum_{k=1}^4 e_k\otimes e_k\right)\cdot \eta_i\otimes \eta_0,\qquad
C_i=\eta_i\otimes \eta_0,
$$
for $i\in\{1,2,3\}$.
Here, we identified the algebra of random variables on $\{1,2,3,4\}^2$ with $\R^4\otimes \R^4$.
The bilocality constraints \eqref{e:biloc} are satisfied because of the tensor structure of $A_i,C_i$ (and the uniform distribution on a product is the product of uniform distributions), and the vanishing constraints \eqref{e:Z=0} follow by direct calculation. Finally, \eqref{e:tgb} evaluates to 4 for this ensemble of binary random variables.

Lastly, let us provide an analytic proof that the value of \eqref{e:tgb} subject to \eqref{e:biloc} and \eqref{e:Z=0} is at most 4.
The sub-cases in the proof were identified by looking at the numerical output of the aforementioned dual SDP for $r=3$.

\begin{proposition}\label{p:exact}
If binary random variables $A_i,B_i,C_i$ for $i=1,2,3$ satisfy constraints \eqref{e:biloc} and \eqref{e:Z=0}, then
$$
\frac13 \sum_{i\in\{1,2,3\} } \Big(\bE(B_iC_i)-\bE(A_iB_i)\Big)-\sum_{\{i,j,k\}=\{1,2,3\} } \bE(A_iB_jC_k)\le 4.
$$
\end{proposition}

\begin{proof}
It suffices to see that
\begin{equation}\label{e:exact1}
\bE(B_iC_i)-\bE(A_iB_i)\le 1
\end{equation}
for $i=1,2,3$, and
\begin{equation}\label{e:exact2}
-\bE(A_iB_jC_k)-\bE(A_kB_jC_i)\le 1
\end{equation}
for $j=1,2,3$ and $\{i,k\}=\{1,2,3\}\setminus \{j\}$.
First, \eqref{e:exact1} follows from
$$1-\bE(B_iC_i)+\bE(A_iB_i)=\bE\Big(
(1-B_iC_i)(1+A_iB_i)
\Big)\ge0;$$
here, the equality holds by $B_i^2=1$ and $\bE(A_iC_i)=\bE(A_i)\bE(C_i)=0$, and the inequality holds since $1-B_iC_i$ and $1+A_iB_i$ are nonnegative random variables.
Second, \eqref{e:exact2} follows from the equality
$$
\resizebox{.99\textwidth}{!}{
$1+\bE(A_iB_jC_k)+\bE(A_kB_jC_i)
=\bE\left(
\Big(
\tfrac12 \Big((A_iA_k+C_iC_k)B_j+A_iC_i+A_kC_k\big)
-\bE(A_iA_k)B_j
\Big)^2
\right),$
}
$$
which is verified by expanding the right-hand side, and applying the relations \eqref{e:biloc} and \eqref{e:Z=0}.
\end{proof}

\section{Lasserre's perturbative Positivstellensatz for moment polynomials}
\label{sec:lass}

This section addresses moment polynomial positivity on general moment semialgebraic sets (without the archimedean assumption from Section \ref{sec:arch}). We show that moment polynomials nonnegative on $\bcK(S_1,S_2)\times K(S_1)$ belong to the quadratic module $\QM{S_1,S_2}$ \emph{up to an arbitrarily small perturbation} of their coefficients (Theorem \ref{t:lass}). This is achieved through the analysis of a sequence of conic optimization problems and their duals, and a sufficient condition for partially estimating a positive functional on $\MP$ with a moment evaluation.
Finally, a corollary for polynomial positivity on semialgebraic sets is given (Corollary \ref{c:LN}).

The aforementioned perturbations arise from polynomials
$$\Phi_r=\sum_{j=1}^n\sum_{k=0}^r \frac{x_j^{2k}}{k!}$$
and pure moment polynomials
$$\Psi_r=\sum_{j=1}^n\sum_{\substack{k,\ell\in\N,\\ k\ell\le r}} \frac{\y(x_j^{2k})^\ell}{(k!)^\ell \ell!}$$
for $r\in\N$.
We start with a crude estimate pertaining to the sequence $(\Psi_r)_r$.

\begin{lemma}\label{l:expexp}
For all $y\ge0$,
$$\sum_{k,\ell\in\N}\frac{y^{k\ell}}{(k!)^\ell\ell!}<\infty.$$
In particular, $\sum_{k,\ell\in\N}\frac{1}{(k!)^\ell\ell!}\le e^2$.
\end{lemma}

\begin{proof}
Since $\exp(t)-1\le t \exp(t)$ for $t\ge0$,
\begin{equation}\label{e:expexp}
\sum_{k,\ell\in\N}\frac{y^{k\ell}}{(k!)^\ell\ell!}
=\sum_{k=1}^\infty\sum_{\ell=1}^\infty
\frac{1}{\ell!}\left(\frac{y^k}{k!}\right)^\ell
=\sum_{k=1}^\infty\left(
\exp\left(\frac{y^k}{k!}\right)-1\right) \le \sum_{k=1}^\infty
\frac{y^k}{k!}\exp\left(\frac{y^k}{k!}\right).
\end{equation}
Since $\lim_{k\to\infty} \sqrt[k]{k!}=\infty$, there exists $k_0\in\N$ such that $\frac{y^k}{k!}\le 1$ for all $k\ge k_0$ (in particular, if $y=1$ one can take $k_0=1$). Then we estimate  \eqref{e:expexp} as
\begin{align*}
\sum_{k=1}^\infty
\frac{y^k}{k!}\exp\left(\frac{y^k}{k!}\right)
&=\sum_{k=1}^{k_0-1}
\frac{y^k}{k!}\exp\left(\frac{y^k}{k!}\right)
+\sum_{k=k_0}^{\infty}
\frac{y^k}{k!}\exp\left(\frac{y^k}{k!}\right) \\
&\le \sum_{k=1}^{k_0-1}
\frac{y^k}{k!}\exp\left(\frac{y^k}{k!}\right)
+\sum_{k=k_0}^{\infty}
\frac{y^k}{k!}\cdot e \\
&\le \sum_{k=1}^{k_0-1}
\frac{y^k}{k!}\exp\left(\frac{y^k}{k!}\right)
+e \cdot \exp(y)<\infty.
\end{align*}
If $y=1$, the first sum above is empty, so \eqref{e:expexp} is bounded by $e^2$.
\end{proof}

If $\uX\in\R^n$, then $n\le\Phi_r(\uX)\le\sum_{j=1}^n\exp(X_j^2)$ is uniformly bounded for all $r\in\N$. Similarly, if $\nu\in\prob{\R^n}$ is finitely supported, i.e., $\nu=\sum_{i=1}^m \alpha_i \delta_{\uX_i}$ is a convex combination of the Dirac delta measures $\delta_{\uX_i}$ concentrated at $\uX_i\in \R^n$, then setting 
$y_j:=X_{1j}^2+\cdots+X_{mj}^2$ gives
$$
0\le \Psi_r(\nu)
=\sum_{j=1}^n\sum_{\substack{k,\ell\in\N,\\ k\ell\le r}} \frac{\left(
\sum_{i=1}^m \alpha_i X_{ij}^{2k}
\right)^\ell}{(k!)^\ell \ell!}
\le \sum_{j=1}^n\sum_{\substack{k,\ell\in\N,\\ k\ell\le r}} \frac{y_j^{k\ell}}{(k!)^\ell \ell!}
\le \sum_{j=1}^n\sum_{k,\ell\in\N} \frac{y_j^{k\ell}}{(k!)^\ell \ell!}
<\infty
$$
by Lemma \ref{l:expexp},
so $\Psi_r(\nu)$ is uniformly bounded for all $r\in\N$.

Let $S_1\subseteq \px$, $S_2\subseteq \mp$ and $f\in \MP$. In the proof of Theorem \ref{t:lass}, the following pair of optimization problems for every $r\ge \frac{\deg f}{2}$ and $M>0$ plays an important role:
\begin{equation}
\label{eq:Q_rm}
Q_{r,M}: \begin{cases}
\begin{aligned}
\sup_{z\in\R} \quad  & z  \\	
\emph{s.t.}
\quad & f-z \in \QM{S_1,S_2}_{2r} 
+ \R_{\geq0} \cdot\big(M-\Phi_r-\Psi_r\big)
\,;
\end{aligned}
\end{cases}
\end{equation}
\begin{equation}
\label{eq:Q*_rm}
Q^\vee_{r,M}: \begin{cases}
\begin{aligned}
\inf_{L\in\MP^\vee_{2r}} \quad  & L(f)  \\	
\emph{s.t.}
\quad & L(1)=1
 \,, \\
\quad & L\big(M-\Phi_r-\Psi_r\big)\geq0
 \,, \\
\quad & L(g)\geq0\quad \text{for all } g\in\QM{S_1,S_2}_{2r}\,.
\end{aligned}
\end{cases}
\end{equation}
Here, $\R_{\geq0} \cdot g$ for $g\in\MP$ denotes the set of all nonnegative multiples of $g$.
In the following two lemmas and their proofs we abbreviate 
$\cC_{r,M}=\QM{S_1,S_2}_{2r} + \R_{\geq0}\cdot (M-\Phi_r-\Psi_r)$.
The next lemma provides information on the closure of $\cC_{r,M}$.

\begin{lemma}\label{l:closed}
For all $r\in\N$ and $M>0$, the closure of the cone $\MP_r\cap\cC_{2^r,M}$ in $\MP_r$ is contained in
\begin{equation}\label{e:closure}
\left\{g\in\MP_r\colon g+\ve\in \cC_{2^r,M} \text{ for all }\ve>0\right\}.
\end{equation}
\end{lemma}

\begin{proof}
Let us fix $r\in\N$, $M>0$ and 
let us endow the finite-dimensional space $\MP_r$ with some norm $\|\cdot\|$. 
Observe the identities
\begin{equation}\label{e:recursive}
\begin{split}
\pm 2x_j^k v&=(x_j^k\pm v)^2- x_j^{2k}- v^2,\\
\pm 2\y_{i_1,\dots,i_n}^k v&=(\y_{i_1,\dots,i_n}^k\pm v)^2
+\left(\y_{2ki_1,\dots,2ki_n}-\y_{i_1,\dots,i_n}^{2k}\right)
-\y_{2ki_1,\dots,2ki_n}
-v^2,\\
\pm 2\y(x_j^kv)&=\y\left((x_j^k\pm v)^2\right)- \y(x_j^{2k})- \y(v^2),
\end{split}
\end{equation}
for all $v\in\MP$, and note that $\y_{2ki_1,\dots,2ki_n}-\y_{i_1,\dots,i_n}^{2k}\in\qm{\emptyset,\emptyset}$ by Lemma \ref{l:holder}(2).
Let $u\in\MP_r$ be an arbitrary monomial. Applying \eqref{e:recursive} recursively to factors of $u$ shows that $\pm u$ equals an element of $\QM{\emptyset,\emptyset}_{2^r}$ minus a conic combination of $x_j^{2k},\y(x_j^{2k})$ for $k\le 2^r$.
Here, the bound $2^r$ arises from $\deg u\le r$ and squaring on the right-hand side of \eqref{e:recursive} when extracting factors of $u$ (cf. Lemma \ref{l:adhoc} below).
Thus there exists $A_u>0$ such that
$$\pm u+A_u \in \QM{\emptyset,\emptyset}_{2^r}+\R_{\ge0}\cdot\big(M-\Phi_{2^r}-\Psi_{2^r})\big).$$
Since $\MP_r$ is finite-dimensional, there exists $A>0$ (dependent on $r$ and $M$) such that
\begin{equation}\label{e:constA}
g+A\|g\| \in \QM{\emptyset,\emptyset}_{2^r}+\R_{\ge0}\cdot\big(M-\Phi_{2^r}-\Psi_{2^r})\big)
\end{equation}
for all $g\in\MP_r$. 
Now define $F:\MP_r\to [-\infty,\infty]$ as
$$F(g) = \sup\left\{z\in\R\colon g-z\in \cC_{2^r,M}
\right\}.$$
This function satisfies the following properties:
\begin{enumerate}[\rm (a)]
\item $F(g)\ge -A\|g\|$ for all $g\in\MP_r$;
\item $F(g_1+g_2)\ge F(g_1)+F(g_2)$ for all $g_i\in\MP_r$;
\item $F(g)\ge0$ if and only if $g$ belongs to \eqref{e:closure}.    
\end{enumerate}
Here, (a) follows from \eqref{e:constA} because $\QM{\emptyset,\emptyset}_{2^r}+\R_{\ge0}\cdot\big(M-\Phi_{2^r}-\Psi_{2^r})\big)$ is contained in $\cC_{2^r,M}$; (b) holds since $g_1-z_1,g_2-z_2 \in\cC_{2^r,M}$ implies $g_1+g_2-(z_1+z_2)\in\cC_{2^r,M}$; and (c) follows directly by definitions of $F$ and \eqref{e:closure}.
Now suppose $(g_i)_i$ is a sequence in $\MP_r\cap\cC_{2^r,M}$ that converges to $g\in\MP_r$. Then
$$F(g) \ge F(g_i)+F(g-g_i)\ge -A\|g-g_i\|$$
for all $i$, and so $F(g)\ge0$. Therefore, $g$ belongs to \eqref{e:closure}.
\end{proof}

\begin{remark}
Note that $\cC_{r,M}$ is not closed in general. Indeed, following \cite[Remark 2.8]{powers01} let $S_1=\{-x_1^2\}$; then $x_1+\ve\in \QM{S_1,\emptyset}=\cC_{1,1}$ for all $\ve>0$ but $x_1\notin \cC_{1,1}$.
\end{remark}

One particular aspect of the proof of Lemma \ref{l:closed} is reused later, so we record it in detail next.

\begin{lemma}\label{l:adhoc}
For all $i_1,\dots,i_n\in\N_0$,
$$
\sum_{j=1}^{n-1}\frac{1}{2^j}\y(x_j^{2^{j+1} i_j})
+\frac{1}{2^{n-1}}\y(x_n^{2^n i_n})
-\y_{i_1,\dots, i_n}^2 \in\qm{\emptyset,\emptyset}.
$$
\end{lemma}

\begin{proof}
By Lemma \ref{l:holder}(2) it suffices to see that
\begin{equation}\label{e:adhoc}
\sum_{j=1}^{n-1}\frac{1}{2^j}\y(x_j^{2^{j+1} i_j})
+\frac{1}{2^{n-1}}\y(x_n^{2^n i_n})
-\y_{2i_1,\dots, 2i_n}\in\qm{\emptyset,\emptyset},
\end{equation}
which we prove by induction on $n$.
If $n=1$, the left-hand side of \eqref{e:adhoc} is 0.
Suppose \eqref{e:adhoc} holds for $n-1$. Then
\begin{align*}
&\sum_{j=1}^{n-1}\frac{1}{2^j}\y(x_j^{2^{j+1} i_j})
+\frac{1}{2^{n-1}}\y(x_n^{2^n i_n})
-\y_{2i_1,\dots, 2i_n}\\
=\,&
\sum_{j=1}^{n-1}\frac{1}{2^j}\y(x_j^{2^{j+1} i_j})
+\frac{1}{2^{n-1}}\y(x_n^{2^n i_n})
+\frac12\y\left(\left(x_1^{2i_1}-x_2^{2i_2}\cdots x_n^{2i_n}\right)^2\right)-\frac12\y(x_1^{4i_1})-\frac12\y_{0,4i_2,\dots,4i_n} \\
=\,&
\frac12\left(
\sum_{j=1}^{n-2}\frac{1}{2^j}\y(x_{j+1}^{2^{j+1} \cdot 2i_{j+1}})
+\frac{1}{2^{n-2}}\y(x_n^{2^{n-1} \cdot 2i_n})
-\y_{0,4i_2,\dots,4i_n}
\right)+\frac12\y\left(\left(x_1^{2i_1}-x_2^{2i_2}\cdots x_n^{2i_n}\right)^2\right)
\end{align*}
belongs to $\qm{\emptyset,\emptyset}$ by the induction hypothesis.
\end{proof}

We can now show that the solutions of \eqref{eq:Q_rm} and \eqref{eq:Q*_rm} converge to the same value.

\begin{lemma}\label{l:dual}
For all $r$ and $M$, the optimization problem \eqref{eq:Q_rm} is a linear conic problem, and \eqref{eq:Q*_rm} is its dual.
Sequences $(\sup Q_{r,M})_r$ and $(\inf Q^\vee_{r,M})_r$ are increasing. 
If $S_2$ is finite, $\bcK(S_1,S_2)\neq\emptyset$ and $M$ is large enough, then
\begin{equation}\label{e:nogap}
\lim_{r\to\infty} \sup Q_{r,M} 
= \lim_{r\to\infty} \inf Q^\vee_{r,M}. 
\end{equation}
\end{lemma}

\begin{proof}
The optimization problems \eqref{eq:Q_rm} and \eqref{eq:Q*_rm} have linear objective functions and constraints that consist of linear equations and cone memberships, so \eqref{eq:Q_rm} and \eqref{eq:Q*_rm} are linear conic problems. Furthermore, they are dual to each other, by the definition of duality \cite[Section IV.6]{Bar}.
The sequence $(\sup Q_{r,M})_r$ is increasing since $\cC_{r,M}\subseteq \cC_{r+1,M}$.
If $L$ is feasible for $Q^\vee_{r,M}$, then its
restriction is feasible for $Q^\vee_{r',M}$ if $r'<r$;
hence the sequence $(\inf Q^\vee_{r,M})_{r\geq d}$ is increasing (the set of constraints on $L$ increases with $r$, and consequently so does the infimum).
By weak duality \cite[Theorem IV.6.2]{Bar} we have
$\sup Q_{r,M}\leq\inf Q^\vee_{r,M}$.
The rest of the proof is dedicated to showing that $\inf Q^\vee_{r,M}\le\sup Q_{2^r,M}$;
then \eqref{e:nogap} follows immediately from 
$\sup Q_{r,M}\le\inf Q^\vee_{r,M}\le\sup Q_{2^r,M}$.
Since $\bcK(S_1,S_2)\neq\emptyset$, there exists $(\nu,\uX)\in \bcK(S_1,S_2)\times K(S_1)$.
By Proposition \ref{p:tchakaloff}, we can further assume that $\nu$ is supported on a finite set.
Let
$M\ge \sup_r (\Phi_r(\uX)+\Psi_r(\nu))$ (note that the right hand side is finite).
Then $L\in\MP_{2r}^\vee$ defined by $L(p)=p(\nu,\uX)$ is clearly feasible for \eqref{eq:Q*_rm},
whence $\inf Q^\vee_{r,M}<\infty$ for all $r\ge \frac{\deg f}{2}$.
Note that $L(f-\inf Q^\vee_{r,M})\geq0$ for all $L\in\cC_{r,M}^\vee$. This implies that $f-\inf Q^\vee_{r,M}$ is in $\cC_{r,M}^{\vee\vee}$, which is the closure of $\cC_{r,M}$. 
Therefore, $f-\inf Q^\vee_{r,M}+\ve\in \cC_{2^r,M}$ for all $\ve>0$ by Lemma \ref{l:closed}, so 
$\inf Q^\vee_{r,M}\le\sup Q_{2^r,M}$.
\end{proof}

Next, we turn to a partial estimate of positive functionals on moment polynomials with moment evaluations (cf.~\cite{infusino2014,alpay2015} for related infinite-dimensional moment problems), in the spirit of Nussbaum's theorem \cite[Theorem 14.25]{schmbook} on functionals satisfying the multivariate Carleman condition. This classical result states that if the marginals of a positive linear functional $L$ on $\px$ do not grow too fast (Carleman's condition, which is in particular satisfied if there is $M>0$ such that $L(x_j^{2k})\le M^k k!$ for all $k$), then $L$ equals integration with respect to some measure on $\R^n$.

\begin{proposition}\label{p:det}
Let $S_1\subseteq\px$, $S_2\subseteq\mp$, $M>0$, and $f\in\MP$. Suppose $L\in\MP^\vee$ satisfies
\begin{enumerate}[\rm (a)]
\item $L(1)=1$,
\item $L(\QM{S_1,S_2})=\R_{\ge0}$,
\item $L(x_j^{2k})\le k!M$ and $L(\y(x_j^{2k})^\ell)\le \ell!(k!)^\ell M$ for all $j\in\{1,\dots,n\}$ and $k,\ell\in\N$. 
\end{enumerate}
Then there exists $(\mu,\uX)\in\bcK(S_1,S_2)\times K(S_1)$ 
such that
$L(f) \ge f(\mu,\uX)$.
\end{proposition}

Before diving into the proof, two remarks are in order. Firstly, Proposition \ref{p:det} does not provide a full estimate of a functional with a moment evaluation; given a functional $L$ with properties (a), (b), (c) as above, it does not provide a point $\uX$ and a measure $\mu$ such that $L(g)\le g(\mu,\uX)$ for \emph{all} $g\in\MP$. Rather, it provides $\uX$ and $\mu$ such that $L(f)\ge f(\mu,\uX)$ for a \emph{single} $f\in\MP$ given in advance.
Secondly, while the proof of Proposition \ref{p:det} below relies on classical ideas from functional analysis and moment problems, it is rather intricate. At its core is a version of the Gelfand-Naimark-Segal (GNS) construction producing $\uX$ and $\mu$. However, handling unboundedness requires special care (hence analyticity of vectors and strong commutation of unbounded operators are invoked), and obtaining $\uX$ and $\mu$ is a two-step process (first constructing a point evaluation on $\MP$ using direct integral decomposition of the GNS representation, and then arguing that its restriction to $\y(\px)$ arises from a measure). For the theory of unbounded operators, which is extraneous to the rest of the paper outside of this proof, we refer the reader to the books \cite{schmOalg,schmUnbded} (see also footnotes for short explanations).

\begin{proof}[Proof of Proposition \ref{p:det}]
We divide the proof into three steps.

\emph{Step 1: unbounded GNS construction}.
On $\MP$ we define a semi-inner product $\langle p,q\rangle =L(pq)$. Let $\cN=\{p\in\MP\colon L(p^2)=0\}$. By the Cauchy-Schwarz inequality for semi-inner products, $\cN$ is an subspace of $\MP$: if $p,q\in \cN$ then $L(pq)^2\le L(p^2)L(q^2)=0$ implies that $L((p+q)^2)=L(p^2)+2L(pq) +L(q^2)=0$, so $p+q\in\cN$. 
Furthermore, if $p\in\cN$ then $p^2\in \cN$. Indeed, if $L(p^2)=0$, then
$$0\le L\left((p\pm \ve p^3)^2\right)
=L(p^2)\pm 2\ve L(p^4)+\ve^2 L(p^6)
=\ve\left(\pm 2 L(p^4)+\ve L(p^6)\right)
$$
and hence $\pm 2 L(p^4)+\ve L(p^6)\ge0$ for every $\ve>0$, implying $L(p^4)=0$. Consequently, $\cN$ is an ideal of $\MP$: if $p\in\cN$ and $q\in\MP$ then $L((pq)^2)^2\le L(p^4)L(q^4)=0$ by the Cauchy-Schwarz inequality, and so $pq\in\cN$.
Let $\cH$ be the completion of the inner product space $\MP/\cN$. Multiplication with generators $x_j$ and $\y_{i_1,\dots,i_n}$ in $\MP$ induces symmetric\footnote{
An operator $S:D\to\cH$ on a dense subspace $D\subseteq\cH$ is \emph{symmetric} if $\langle Su,v\rangle=\langle u,Sv\rangle$ for all $u,v\in D$.
} unbounded operators $X_j$ and $Y_{i_1,\dots,i_n}$ on $\cH$ with a dense domain $\MP/\cN$. 
Moreover, the Cauchy-Schwarz inequality and the condition (c) on $L$ imply that for all $j\in\{1,\dots,n\}$ and $k,\ell\in\N$, the inequalities
\begin{align*}
\|X_j^k p\|^2&={L(x_j^{2k} p^2)}\le \sqrt{L(x_j^{4k})}\sqrt{L(p^4)}\le 
\sqrt{(2k)!} \sqrt{M}{\|p^2\|} \\
\|Y_{k e_j}^\ell p\|^2&={L(\y(x_j^k)^{2\ell} p^2)}
\le \sqrt{L(\y(x_j^{2k})^{4\ell}}\sqrt{L(p^4)} 
\le \sqrt{(4\ell)!}\sqrt{(k!)^{4\ell} M}{\|p^2\|}
\end{align*}
hold for all $p\in\MP/\cN$ and $k,\ell\in\N$ 
(here, $Y_{k e_j}=Y_{\dots,0,k,0\dots}$ with $k$ at the $j$th position is the operator given by multiplication with $\y(x_j^k)$).
By induction on $k$ it is easy to see that $\sqrt[m]{(mk)!}\le m^k k!$ for all $k,m\in\N$. Therefore,
\begin{align*}
\|X_j^k p\|&\le \sqrt{k!}\sqrt{2}^k\sqrt[4]{M}\sqrt{\|p^2\|}, \\
\|Y_{k e_j}^\ell p\|&\le 
\ell!(4\cdot k!)^\ell\sqrt[4]{M}\sqrt{\|p^2\|}
\end{align*}
for all $k,\ell,j$ and $p\in\mx/\cN$.
In terms of \cite[Definition 7.1]{schmUnbded}, the elements of $\MP/\cN$ are analytic\footnote{
A vector $v\in D$ is \emph{analytic} for $S:D\to D$ if there is $M>0$ such that $\|S^kv\|\le M^k k!$ for all $k$.
} vectors for $X_j$ and $Y_{k e_j}$, for every $j\in\{1,\dots,n\}$ and $k\in\N$.
By \cite[Theorem 7.18]{schmUnbded}, the closures\footnote{
The \emph{closure} of an operator $T$ on $\cH$ is the operator on $\cH$ whose graph in $\cH\oplus\cH$ is the closure of the graph of $T$. Every densely defined symmetric operator admits the closure.
} $\overline{X}_j$ and $\overline{Y}_{k e_j}$ are pairwise strongly commuting self-adjoint operators\footnote{
An operator is \emph{self-adjoint} if it is equal to its own adjoint on $\cH$ (this notion is stronger than being symmetric).
\emph{Spectral projections} of a self-adjoint operator $S$ are projections on $\cH$ corresponding to measurable subsets of the spectrum of $S$ via the spectral theorem.
Two self-adjoint operators \emph{strongly commute} if all their spectral projections commute.
}.
In particular, $\overline{Y}_{2k e_j}$ are positive self-adjoint operators. By Lemma \ref{l:adhoc} and the Cauchy-Schwarz inequality we have
\begin{equation}\label{e:major}
\begin{split}
\|Y_{i_1,\dots,i_n} p\|^2
&=L(\y_{i_1,\dots,i_n}^2p^2)\\
&\le 
\sum_{j=1}^{n-1}\frac{1}{2^j}L\left(\y(x_j^{2^{j+1} i_j})p^2\right)
+\frac{1}{2^{n-1}}L\left(\y(x_n^{2^n i_n})p^2\right) \\
&\le
\sum_{j=1}^{n-1}\frac{1}{2^j}\left\|\left(\overline{Y}_{(2^{j+1} i_j)e_j}+I\right) p\right\|^2
+\frac{1}{2^{n-1}}\left\|\left(\overline{Y}_{(2^n i_n)e_n}+I\right) p\right\|^2
\end{split}
\end{equation}
for all $p\in\MP/\cN$. 
That is, the operators $Y_{i_1,\dots,i_n}$ are majorized in the sense of \eqref{e:major} by self-adjoint operators that pairwise strongly commute with all $\overline{X}_j$ and $\overline{Y}_{k e_j}$.
Thus, for every collection $i_1,\dots,i_n,i_1',\dots,i_n',j$ there exists a self-adjoint operator $S$ on $\cH$
that commutes with $X_j,Y_{i_1,\dots,i_n},Y_{i_1',\dots,i_n'}$ on $\MP/\cN$, and satisfies
$$
\|X_jp\|,\, \|Y_{i_1,\dots,i_n}p\|,\, \|Y_{i_1',\dots,i_n'}p\| \le \|Sp\|
$$
for all $p\in\MP/\cN$. Then the closures $\overline{X}_j,\overline{Y}_{i_1,\dots,i_n},\overline{Y}_{i_1',\dots,i_n'}$ are pairwise strongly commuting self-adjoint operators by \cite[Proposition 2]{schm88}.
Since real polynomials in strongly commuting self-adjoint operators are again self-adjoint,
$$\varphi(x_j)=\overline{X}_j,\quad \varphi(\y_{i_1,\dots,i_n})=\overline{Y}_{i_1,\dots,i_n}$$
defines an integrable representation\footnote{
An \emph{integrable representation} of a commutative algebra (with trivial involution) is a homomorphism into a subalgebra of strongly commuting self-adjoint operators on a Hilbert space.
} $\varphi$ of $\MP$ on $\cH$ according to \cite[Definition 9.1.1]{schmOalg}.  By the construction, $\varphi(\QM{S_1,S_2})$ consists of positive self-adjoint operators.

\emph{Step 2: integral decomposition of $\varphi$}.
Irreducible integral representations of $\MP$ are one-dimensional by the unbounded analog of Schur's lemma\footnote{
The Schur-Dixmier lemma implies that irreducible representations of an abelian $\C$-algebra of countable dimension are one-dimensional.
} \cite[Corollary 9.1.11]{schmOalg}. 
Since $\MP$ is countably generated, by \cite[Theorem 12.3.5 and subsequent Remark 2]{schmOalg} there exist a compact metric space $T$ with a Borel probability measure $\lambda$ and a measurable family of integral representations $\varphi_t$ of $\MP$ acting on Hilbert spaces $\cH_t$ for $t\in T$, such that $\varphi$ is unitarily equivalent to the direct integral\footnote{
\emph{Direct integrals} of Hilbert spaces, operators and representations generalize direct sums of these objects. While infinite-dimensional integral representations do not necessarily admit a direct sum decomposition into irreducible ones, they do admit a direct integral decomposition into irreducible ones.
}
$$\int_T^\oplus \varphi_t \d \lambda(t)$$
on the Hilbert space $\cH'=\int_T^\oplus \cH_t \d\lambda(t)$, and almost every (relative to $\lambda$) $\cH_t$ is one-dimensional.
Let $e_t\in\cH_t$ be such that that the vector $\int_T^\oplus e_t\d\lambda(t)\in\cH'$ corresponds to the unit vector $1\in\cH$ under the aforementioned unitary equivalence. 
Since $1\in\cH$ is a cyclic\footnote{
A vector $v\in\cH$ is \emph{cyclic} for an algebra of operators $\mathcal{A}$ on $\cH$ if $\mathcal{A}\cdot v$ is dense in $\cH$.
} vector for $\varphi(\MP)$, we have $e_t\neq0$ for almost all $t$.
Define $L_t:\MP\to \R$ as $L_t(p)=\langle \varphi_t(p)e_t,e_t\rangle$; then
\begin{equation}\label{e:dirint}
L(p)=\int_T L_t(p)\d\lambda(t)
\end{equation}
for all $p\in\MP$.
Since $\varphi(\QM{S_1,S_2})$ consists of positive self-adjoint operators, it follows by \cite[Proposition 12.2.3(ii)]{schmOalg} that $\varphi_t(\QM{S_1,S_2})$ consists of positive selfadjoint operators, for almost every $t$. In particular, $L_t(\QM{S_1,S_2})=\R_{\ge0}$ for almost every $t$. 
The subset
$$Z'=\{t\in T\colon L_t(f)\le L(f)L_t(1)\}$$
of $T$ is measurable. Note that $\lambda(Z')=0$ would by \eqref{e:dirint} give rise to the contradiction
$$L(f)=\int_T L_t(f)\d\lambda(t)> \int_T L(f)L_t(1)\d\lambda(t)=L(f)L(1)=L(f),$$
so $m:=\lambda(Z')>0$.

Let $N>0$ be sufficiently large so that $\sum_{k=1}^\infty \frac{n M}{N^k}<m$, and consider the sets
$$Z_{j,k}=\left\{t\in T\colon L_t\big(\y(x_j^{2k})\big)\ge N^k k!\right\}$$
for $j\in\{1,\dots,n\}$ and $k\in\N$. Then $Z_{j,k}$ is measurable, and
$$
k!M\ge L\big(\y(x_j^{2k})\big) 
=\int_T L_t\big(\y(x_j^{2k})\big)\d\lambda(t) 
\ge\int_{Z_{j,k}} L_t\big(\y(x_j^{2k})\big)\d\lambda(t) 
\ge N^k k!\lambda(Z_{j,k}),
$$
so $\lambda(Z_{j,k})\le \frac{M}{N^k}$. By the choice of $N$,
$$\lambda\left(\bigcup_{j=1}^n\bigcup_{k\in\N} Z_{j,k}\right)\le 
\sum_{j=1}^n\sum_{k=1}^\infty\lambda(Z_{j,k})\le 
n\sum_{k=1}^\infty\frac{M}{N^k}<m,
$$
and so
$$\lambda\left(Z'\setminus\bigcup_{j,k} Z_{j,k}\right)>0.$$
In particular, there exists $t\in T$ such that $\varphi_t$ is one-dimensional, $e_t\neq0$, and
$$L_t(\QM{S_1,S_2})=\R_{\ge0},
\quad L_t(f)\le L(f)L_t(1),
\quad L_t\big(\y(x_j^{2k})\big)\le N^kk!\text{ for }k\in\N.$$
Since $\varphi_t:\MP\to\R$ is a nonzero homomorphism, it is unital, thus $L_t(p)=\varphi_t(p)L_t(1)$ for all $p\in\MP$. Since $L_t(1)=\|e_t\|^2>0$,
$$\varphi_t(\QM{S_1,S_2})=\R_{\ge0},
\quad \varphi_t(1)= 1,
\quad \varphi_t(f)\le L(f),
\quad \varphi_t\big(\y(x_j^{2k})\big)\le \frac{N^kk!}{L_t(1)}\text{ for }k\in\N.$$

\emph{Step 3: conclusion}.
Consider the functional $\widetilde L:\px\to\R$ given by $\widetilde L(p)=\varphi_t(\y(p))$. 
Then $\widetilde L(\QM{S_1})=\R_{\ge0}$, and $\widetilde L$ satisfies the Carleman condition since $\widetilde L(x_j^{2k})\le \frac{1}{L_t(1)}N^k k!$ for all $j=1,\dots,n$ and $k\in\N$.
By a refined version of Nussbaum's theorem \cite[Theorem 14.25]{schmbook} applied to $\widetilde L$, there is $\mu\in\bcK(S_1,S_2)$ such that $\varphi_t(\y_{i_1,\dots,i_n})=\widetilde L(x_1^{i_1}\cdots x_n^{i_n})=\int x_1^{i_1}\cdots x_n^{i_n} \d\mu$ for all $(i_1,\dots,i_n)\in \N_0^n$. Let $\uX=(\varphi_t(x_1),\dots,\varphi_t(x_n))$; then $\uX\in K(S_1)$ because $\varphi_t(\QM{S_1,S_2})=\R_{\ge0}$. Lastly, $f(\mu,\uX)=\varphi_t(f)\le L(f)$.
\end{proof}

The following auxiliary lemma bounds the growth of a positive linear functional on moment polynomials in terms of its marginal values (on $x_j^{2k}$ and $\y(x_j^{2k})$).

\begin{lemma}\label{l:tech}
Let $r\in\N$, $M\ge1$ and $L\in \MP^\vee_{2r}$. Suppose
$L(\QM{\emptyset,\emptyset}_{2r})=\R_{\ge0}$ and
\begin{equation}\label{e:growth0}
L\left( x_j^{2k}\right),L\left( \y(x_j^{2k})\right)\leq k!\,M
\end{equation}
for $j=1,\dots,n$ and $k=1,\dots,r$. Then
\begin{equation}\label{e:growth1}
|L(w)|\leq \sqrt{(\deg w)!\,M}
\end{equation}
for all monomials $w$ in $\MP_r$. 
\end{lemma}

\begin{proof}
By applying \cite[Lemma 6.2]{las06perturb} to the moment matrix $(L(\y(\alpha\beta)))_{\alpha,\beta}$ indexed by $x_1^{i_1}\cdots x_n^{i_n}$ for $i_1+\cdots+i_n\le r$, one obtains
\begin{equation}\label{e:QrMbd1}
L(\y_{2i_1,\dots,2i_n})\leq (i_1+\cdots+i_n)!\,M
\end{equation}
for $i_1+\cdots+i_n\le r$. 
Next,
\begin{equation}\label{e:QrMbd2}
L(\y_{i_1,\dots,i_n}^{2k})
\leq L(\y_{2k i_1,\dots,2k i_n})
\end{equation}
for $k(i_1+\cdots+i_n)\le r$ since $\y_{2k i_1,\dots,2k i_n}-\y_{i_1,\dots,i_n}^{2k}
\in\qm{\emptyset,\emptyset}$ 
by Lemma \ref{l:holder}(2).
By \eqref{e:QrMbd1}, \eqref{e:QrMbd2} and \cite[Lemma 6.2]{las06perturb}, 
$$L(w^2)
\leq (\deg w)!\,M
$$
for all monomials $w\in\MP_r$.
Finally, \eqref{e:growth1} follows from $L(w^2)-L(w)^2=L((w-L(w))^2)\ge0$ for $\deg w\le r$.
\end{proof}

In the final preparation for Theorem \ref{t:lass}, we connect the solution of the optimization problem \eqref{eq:Q*_rm} with the infimum of $f$ on $\bcK(S_1,S_2)\times K(S_1)$.

\begin{lemma}\label{l:upper}
Suppose $S_2$ is finite, $\bcK(S_1,S_2)\neq\emptyset$ and $f\in\MP$ is bounded below on $\bcK(S_1,S_2)\times K(S_1)$; denote
$f_*:=\inf_{(\mu,\uX)\in\bcK(S_1,S_2)\times K(S_2)} f(\mu,\uX) > -\infty$.
For large enough $M>0$,
\eqref{eq:Q*_rm} is feasible for $2r\geq \deg f$,
and $\inf Q^\vee_{r,M} \nearrow f_M$ as $r\to\infty$
for some $f_M\geq f_*$.
\end{lemma}

\begin{proof}
Feasibility of \eqref{eq:Q*_rm} follow by the same argument as in the proof of Lemma \ref{l:dual}.
Let $L$ be feasible for $Q^\vee_{r,M}$. 
Then $L(M-\Phi_r-\Psi_r)\ge0$ implies
\begin{equation}\label{e:boundterms}
L(x_j^{2k})\le k!\,M\quad\text{for }k\le r,\qquad
L(\y(x_j^{2k})^\ell)\le\ell!\,(k!)^\ell\, M \quad\text{for } k\ell\le r.
\end{equation}
Let $d\in\N$, and $r\ge d$. By Lemma \ref{l:tech} and \eqref{e:boundterms} (for $\ell=1$),
\begin{equation}\label{e:QrMbd}
|L(w)|\le \sqrt{d!\, M}=:c_d
\end{equation}
for all monomials $w$ in $\MP_d$.
In particular, $L(f)$ is uniformly bounded for large enough $r$. 
Therefore, $(\inf Q^\vee_{r,M})_r$ is an increasing function bounded from above,
whence $\inf Q^\vee_{r,M} \nearrow f_M$ as $r\to\infty$, for some $f_M$.
It remains to show $f_M\geq f_*$.

Let $\ell^\infty$ be the space of bounded functions on monomials in $\MP$.
For every $r\in\N$ let $L^{(r)}$ be an optimizer of \eqref{eq:Q*_rm}, and let $s_r\in\ell^\infty$ be given as $s_r(w)=\frac{1}{c_{\deg w}}L^{(r)}(w)$ for monomials $w$ with $\deg w\le 2r$, and $s_r(w)=0$ for all other monomials $w$. 
Note that for every monomial $w$, $s_r(w)$ is bounded by 1 for all sufficiently large $r$.
By the Banach-Alaoglu theorem \cite[Theorem III.2.9]{Bar}, 
the closed unit ball in $\ell^\infty$ is compact in the weak-* topology. In particular, the sequence $(s_r)_r$ has an accumulation point in $\ell^\infty$ with respect to the weak-* topology. Hence, there
is $s\in\ell^\infty$ and a subsequence $(s_{r_m})_m$ converging to $s$. 
Define
$$L:\MP\to\R,\qquad L(w) = c_{\deg w}\cdot s(w).$$
Then $L^{(r_m)}|_{\MP_d}\to L|_{\MP_d}$ as $m\to \infty$, for every $d\in\N$.
In particular, $L$ is a unital linear functional, $L(f)=f_M$, and
$L(\QM{S_1,S_2})=\R_{\ge0}$.
Let $d\in\N$ be arbitrary. Then for every $r_m\ge d$,
$$L(x_j^{2k})\le k!\,M\quad\text{for }k\le r_m,\qquad
L(\y(x_j^{2k})^\ell)\le\ell!\,(k!)^\ell\, M \quad\text{for } k\ell\le r_m$$
by \eqref{e:boundterms} since $L^{(r_m)}$ is feasible for $Q^\vee_{r_m,M}$.
Consequently,
$$L(x_j^{2k})\le k!\,M\quad\text{for }k\in\N,\qquad
L(\y(x_j^{2k})^\ell)\le\ell!\,(k!)^\ell\, M\quad\text{for } k,\ell\in\N.$$
Therefore, $L(f)\ge f_*$
by Proposition \ref{p:det}.
\end{proof}

The following is a moment polynomial analog of Lasserre's Positivstellensatz \cite{las06perturb,LasserreNetzer}.
Its proof relies on the interplay of solutions of optimization problems \eqref{eq:Q_rm} and \eqref{eq:Q*_rm} with the infimum of $f$ on $\bcK(S_1,S_2)\times K(S_1)$.

\begin{theorem}[Perturbative Positivstellensatz]\label{t:lass}
Let $S_1\subseteq\px$ and $S_2\subseteq\mp$, and suppose $S_2$ is finite. 
The following statements are equivalent for $f\in\MP$:
\begin{enumerate}[\rm (i)]
    \item $f\ge0$ on $\bcK(S_1,S_2)\times K(S_1)$;
    \item for every $\ve>0$ there exists $r\in\N$ such that $f+\ve (\Phi_r+\Psi_r) \in \QM{S_1,S_2}$.
\end{enumerate}
The following statements are equivalent for $f\in\mp$:
\begin{enumerate}[\rm (i')]
	\item $f\ge0$ on $\bcK(S_1,S_2)$;
	\item for every $\ve>0$ there exists $r\in\N$ such that $f+\ve (1+\Psi_r) \in \qm{S_1,S_2}$.
\end{enumerate}
\end{theorem}

\begin{proof}
(ii)$\Rightarrow$(i) Let $\uX\in K(S_1)$ be arbitrary, and let $\nu$ be a finitely supported measure in $\bcK(S_1,S_2)$. There is $0<M<\infty$ such that
$$\Phi_r(\uX)+\Psi_r(\nu)\le M$$
for all $r\in\N$.
Then for every $\ve>0$ one has $f(\nu,\uX)\ge -\ve M$, and so $f(\nu,\uX)\ge0$.
Since $\uX$ and $\nu$ were arbitrary, and finitely supported measures in $\bcK(S_1,S_2)$ interpolate any measure in $\bcK(S_1,S_2)$ up to moments of any fixed order by Proposition \ref{p:tchakaloff}, it follows that $f(\mu,\uX)\ge0$ for all $(\mu,\uX)\in \bcK(S_1,S_2)\times K(S_1)$.

(i)$\Rightarrow$(ii) We divide the proof into two main cases (a) and (b), according to whether $\bcK(S_1,S_2)$ is empty or not. 

Case (a): assume $\bcK(S_1,S_2)\neq\emptyset$ and denote 
$f_*=\inf_{(\mu,\uX)\in\bcK(S_1,S_2)\times K(S_1)} f(\mu,\uX)$.
We further divide this case in two sub-cases.

First suppose $f_*>0$. By Proposition \ref{p:tchakaloff}, there exists $(\nu,\uX)\in \bcK(S_1,S_2)\times K(S_1)$, with a finitely supported $\nu$. Denote $M_0:=\sup_r (\Phi_r(\uX)+\Psi_r(\nu))<\infty$, and let $M > \max\{\frac{1}{f_*},M_0\}$ be arbitrary. 
By Lemmas \ref{l:dual} and \ref{l:upper}, there exists $r_M>0$ such that
$\sup Q_{r_M,M}> f_*-\frac1M$.
That is, there are $z_M\ge f_*-\frac1M$, $\lambda_M\in\R_{\ge0}$ and $q_M\in \QM{S_1,S_2}_{2r_M}$ such that
\begin{equation}\label{e:sdpsol}
f-z_M=q_M+\lambda_M\big(M-\Phi_{r_M}-\Psi_{r_M}\big).
\end{equation}
Evaluating \eqref{e:sdpsol} at $(\nu,\uX)\in \bcK(S_1,S_2)\times K(S_1)$
gives
\begin{align*}
f(\nu,\uX)-f_*+\frac1M
&\ge f(\nu,\uX)-z_M \\
&= q_M(\nu,\uX)+ \lambda_M \big(M-\Phi_{r_M}(\uX)-\Psi_{r_M}(\nu)\big) \\
&\ge\lambda_M (M-M_0),
\end{align*}
and therefore
\begin{equation}\label{e:lamest}
\lambda_M\le \frac{f(\nu,\uX)-f_*+\frac1M}{M-M_0}.
\end{equation}
The right-hand side of \eqref{e:lamest} goes to $0$ as $M\to\infty$.
By \eqref{e:sdpsol},
$$f+\lambda_M\big(\Phi_{r_M}+\Psi_{r_M}\big)
=z_M+q_M+\lambda_M M \in \QM{S_1,S_2}_{2r_M},$$
and $\lambda_M\to0$ as $M\to \infty$, so (ii) holds.

Now suppose $f_*=0$, and let $\ve>0$ be arbitrary.
By applying (i)$\Rightarrow$(ii) to $f+\frac{n\ve}{2}$ and $\frac{\ve}{2}>0$,
there exists $r\in\N$ such that
$(f+\frac{n\ve}{2})+\frac{\ve}{2} (\Phi_r+\Psi_r)\in \QM{S_1,S_2}_{2r}$.
But the latter equals $f+\ve (\Phi_r+\Psi_r) -\frac{\ve}{2} (\Phi_r-n+\Psi_r)$,
so $f+\ve (\Phi_r+\Psi_r)\in \QM{S_1,S_2}_{2r}$.

Case (b): assume $\bcK(S_1,S_2)=\emptyset$, and let $f\in\px$ and $\ve>0$ be arbitrary. 
Let $x_{n+1}$ be an auxiliary variable, and consider 
$S_1'=x_{n+1}\cdot (\{1\}\cup S_1)\subset \R[x_1,\dots,x_{n+1}]$ and $S_2'=\y_{1,0,\dots,0}^2\cdot S_2\subset \mp$.
Then $K(S_1')$ contains $\R^n\times\{0\}$, $\bcK(S_1',S_2')$ contains all $\mu\in \prob{\R^n\times\{0\}}$ such that $\int x_1\d\mu=0$ (and is thus nonempty),
and $x_{n+1}f\ge0$ on $\bcK(S_1',S_2')$.
By the case (a) of the proof above, there exists $r\in\N$ such that
\begin{equation}\label{e:addvar}
x_{n+1}f+\frac{\ve}{n+e+e^2}
\left(\Phi_r+\Psi_r+\sum_{k=0}^r\frac{x_{n+1}^{2k}}{k!}+\sum_{\substack{k,\ell\in\N,\\ k\ell\le r}} \frac{\y(x_{n+1}^{2k})^\ell}{(k!)^\ell \ell!}\right) 
\in \QM{S_1',S_2'}.
\end{equation}
Consider the homomorphism $\xi$, from moment polynomials generated by $x_1,\dots,x_{n+1}$ to moment polynomials generated by $x_1,\dots,x_n$, that is determined by
$$\xi(x_j)=x_j \text{ for }j\le n,\quad \xi(x_{n+1})=1,\quad \xi(\y_{i_1,\dots,i_{n+1}})=\y_{i_1,\dots,i_{n}}.$$
Note that $\xi$ intertwines with $\y$.
Applying $\xi$ to \eqref{e:addvar} thus gives
\[
f+\frac{\ve}{n+e+e^2}\left(\Phi_r+\Psi_r+
\sum_{k=0}^r\frac{1}{k!}+\sum_{\substack{k,\ell\in\N,\\ k\ell\le r}} \frac{1}{(k!)^\ell \ell!}
\right) 
\in \QM{S_1,S_2'}\subseteq \QM{S_1,S_2},
\]
and therefore
$f+\ve(\Phi_r+\Psi_r)\in \QM{S_1,S_2}$ because $\sum_{k=0}^r\frac{1}{k!}\le e$ and $\sum_{k\ell\le r} \frac{1}{(k!)^\ell \ell!}\le e^2$ by Lemma \ref{l:expexp}.

(i')$\Leftrightarrow$(ii') 
The proof is analogous to (i)$\Leftrightarrow$(ii), and utilizes the straightforward counterparts of Lemmas \ref{l:dual}, \ref{l:tech}, \ref{l:upper} and Proposition \ref{p:det} for $\qm{S_1,S_2}$.
\end{proof}

\begin{remark}\label{r:larger}
In Theorem \ref{t:lass}, one can replace the cone $\QM{S_1,S_2}$ with the larger cone $\QQM{S_1,S_2}$ from Section \ref{s:mompop}. While the resulting statement is slightly weaker than Theorem \ref{t:lass}, it has the advantage that it is more amenable for computations. Namely, for every $r\in\N$, the smallest $\ve_r\ge0$ such that $f+\ve_r(\Phi_r+\Psi_r)\in \QQM{S_1,S_2}_{2r}$ can be computed via SDP. Theorem \ref{t:lass} then implies that $\lim_{r\to \infty}\ve_r=0$ if $f\ge0$ on $\bcK(S_1,S_2)\times K(S_1)$.
\end{remark}

\begin{example}
The implications (ii)$\Rightarrow$(i) and (ii')$\Rightarrow$(i') of Theorem \ref{t:lass} fail in general when $S_2$ is not finite.
Let $n=1$, $f=-1$ and
$$S_2=\{\y_{2i}-(4i+1)!\colon i\in\N \}.$$
Since the $2i$\textsuperscript{th} moment of $\mu =e^{-\sqrt{|t|}}\d t$ is $(4i+1)!$, we have $\bcK(\emptyset,S_2)\neq\emptyset$ and therefore (i') and (i) are false.
Now let $\ve>0$ be arbitrary; then there exists $r\in\N$ such that
$r!\le \ve (4r+1)!$, and so
$$-1+\ve \Psi_r= \frac{\ve}{r!}\left(\y_{2r}-\frac{r!}{\ve}\right)+\ve\left(\Psi_r-\frac{1}{r!}\y_{2r}\right) \in \qm{\emptyset,S_2}.$$
Thus, (ii) and (ii') are true.
\end{example}

\begin{example}
Let $f=\y_{2,0}\y_{0,2}$. Since $f$ is a product of two elements in $\qm{\emptyset,\emptyset}$, it is nonnegative on $\prob{\R^2}$; on the other hand, $f$ does not belong to $\qm{\emptyset,\emptyset}$.
For a fixed $r\in\N$, searching for the smallest $\ve_r\ge0$ such that $f+\ve_r (1+\Psi_r)\in\qm{\emptyset,\emptyset}$ can be formulated as an SDP. For small values of $r$ one obtains $\ve_2=0.33333$, $\ve_3=0.06330$, $\ve_4=0.01416$.
\end{example}

\begin{example}
Let us return to $f=\y_{4,2}\y_{2,4}-\y_{2,2}^3$ from Example \ref{ex:h17}. 
Then $f$ is nonnegative on $\prob{\R^2}$, but
$\rho(f)<0$ for some pseudo-moment evaluation $\rho:\mp\to\R$, so $f$ does not admit a sum-of-squares certificate with denominators in the sense of Hilbert's 17th problem by Theorem \ref{t:h17pseudo}.
Nevertheless, for every $\ve>0$, Theorem \ref{t:lass} guarantees an $r\in\N$ such that $f+\ve (1+\Psi_r) \in \qm{\emptyset,\emptyset}$.
Alternatively, since $f$ is homogeneous with respect to the degree on $\mp$, its nonnegativity on $\prob{\R^2}$ is equivalent to nonnegativity on $\prob{[-1,1]^2}$. By Theorem \ref{t:arch}, $f+\ve\in\qm{\{1-x_1,1+x_1,1-x_2,1+x_2\},\emptyset}$ for every $\ve>0$.
\end{example}

\subsection{Polynomial positivity on arbitrary semialgebraic sets}

Theorem \ref{t:lass} also carries implications for classical (non-moment) polynomials (see Remark \ref{r:LN} for comparison with earlier results).

\begin{corollary}\label{c:LN}
Let $S\subseteq\px$. Then the following statements are equivalent for $f\in\px$:
\begin{enumerate}[\rm (i)]
    \item $f\ge0$ on $K(S)$;
    \item for every $\ve>0$ there exists $r\in\N$ such that $f+\ve \Phi_r\in \QM{S}$.
\end{enumerate}
\end{corollary}

\begin{proof}
The homomorphism $\zeta:\MP\to\px$ determined by $\zeta|_{\px}=\id_{\px}$ and $\zeta(\y_{i_1,\dots,i_n})=0$ (for $i_j$ not all zero) maps $\QM{S,\emptyset}$ into $\QM{S}$. Applying $\zeta$ to the conclusions of Theorem \ref{t:arch} for $f$ and $\QM{S,\emptyset}$ gives the desired statement.
\end{proof}

\begin{remark}\label{r:LN}
The special case of Corollary \ref{c:LN} for $S=\emptyset$ is given in \cite[Theorem 4.1]{las06perturb}.
Furthermore, the conclusion of Corollary \ref{c:LN} has been established in \cite[Corollary 3.7]{LasserreNetzer} under additional assumptions, namely that $S$ is finite and has the strong moment property\footnote{
$S$ has the \emph{strong moment property} if every functional on $\px$ that is nonnegative on the preordering generated by $S$, equals integration with respect to a measure supported on $K(S)$. This is a rather restrictive property; e.g., $S=\{x_1,x_2\}$ does not have it.
}, and the interior of $K(S)$ is nonempty; furthermore, \cite[Corollary 3.7]{LasserreNetzer} requires preorderings instead of quadratic modules.
Corollary \ref{c:LN} disposes of all these assumptions.
On the other hand, \cite{las06perturb} and \cite{LasserreNetzer} provide stronger certificates in two special cases: when nonnegativity on $K(S)\cap[-1,1]^n$ is considered (and $S$ is finite with the strong moment property, and $K(S)$ has nonempty interior), a simpler perturbation can be used in place of $\Phi_r$ \cite[Corollary 3.6]{LasserreNetzer}; when $S$ consists of concave polynomials and $f$ is convex, only sums of squares and conic combinations of $S$ are needed in place of $\QM{S}$ \cite[Corollary 4.3]{las06perturb}.
\end{remark}

\begin{remark}
Corollary \ref{c:LN} characterizes polynomial positivity on arbitrary basic closed semialgebraic sets (and even more general sets, since infinitely many constraints are allowed), but differently from the renowned Krivine-Stengle Positivstellensatz \cite[Theorem 2.2.1]{marshallbook}; while the latter certificate involves preorderings and denominators, the former involves quadratic modules and coefficient perturbations.

\end{remark}

\begin{example}\label{ex:prod}
Let us record one of the simplest cases to which \cite[Corollary 3.7]{LasserreNetzer} does not apply. Clearly, $x_1x_2\ge0$ on $K(\{x_1,x_2\})$. By Corollary \ref{c:LN}, for every $\ve>0$ there exists $r\in\N$ such that $x_1x_2+\ve\Phi_r\in \QM{\{x_1,x_2\}}$, even though the set $\{x_1,x_2\}$ does not have the strong moment property.

For a fixed $r\in\N$, one can find the smallest $\ve_r\ge0$ such that $x_1x_2+\ve_r \y(\Phi_r)\in\qm{\emptyset,\emptyset}$ by solving an SDP. For $r=2,\dots,8$ the values of $\ve_r$ are
$$0.5,\, 0.012428,\, 0.002016,\, 0.000580, 0.000238,\, 0.000117,\, 
0.000065,\, 0.000032.$$
\end{example}

\begin{remark}
As is evident from the proof of Theorem \ref{t:lass}, the sequence of polynomials $\Phi_r$ can be replaced by
\begin{equation}\label{e:other}
\sum_{j=1}^n\sum_{k=0}^r \frac{x_j^{2k}}{c_k}\qquad \text{for }r\in\N,
\end{equation}
where $c_k>0$ are such that $(c_k)_k$ has super-exponential growth (to ensure point-wise convergence of \eqref{e:other}, which is used for (ii)$\Rightarrow$(i) of Theorem \ref{t:lass} and for feasibility of \eqref{eq:Q_rm} and \eqref{eq:Q*_rm}) and $(k^{-k}c_k)_k$ has at most exponential growth (which is needed for applying Proposition \ref{p:det}).
One might further contemplate whether only the constant term and the leading terms of $\Phi_r$ are essential in Corollary \ref{c:LN}; this is indeed true in certain cases \cite[Example 7.9]{KSV}. However, the following example shows this is not true in general.

Let $n=1$, $f=-2$ and $S=\{-1+\frac{x_1^{2k}}{k!-1}\colon k\ge 2\}$. Then $K(S)=\emptyset$ and $f\ge0$ on $K(S)$. 
We claim that $-2+(1+\frac{x_1^{2r}}{r!})\notin \QM{S}$ for every $r\in\N$. Indeed, suppose
\begin{equation}\label{e:bad}
-1+\frac{x_1^{2r}}{r!} = \sigma_1+\sum_{k=2}^\ell \sigma_k\cdot\left(-1+\frac{x_1^{2k}}{k!-1}\right)
\end{equation}
where $\ell\ge2$ and $\sigma_k\in\px$ are sums of squares. Note that $\ell\le r$. Let $X=\sqrt[2\ell]{\ell!-1}$. Then the right-hand side of \eqref{e:bad} is nonnegative at $X$, while the left-hand side of \eqref{e:bad} is negative at $X$, a contradiction.

On the other hand, if $\ve>0$ is arbitrary and $r\ge \frac{2}{\ve}-1$, then
$$-1+\ve\Phi_r=\left(-1+\ve\left(1+\sum_{k=2}^r \frac{k!-1}{k!}\right)\right)+\ve x_1^2+\ve\sum_{k=2}^r \frac{k!-1}{k!}\left(-1+\frac{x_1^{2k}}{k!-1}\right)\in \QM{S},$$
as anticipated by Corollary \ref{c:LN}.
\end{remark}

We conclude the section with a modified Lasserre's SDP hierarchy, applicable to arbitrary semialgebraic sets. 
Let $S\subseteq \px$, $f\in\px$ and $\ve>0$. 
Let $f_*=\inf_{\uX\in K(S)}f(\uX)$.
For $r\ge\frac{\deg f}{2}$ consider the SDP
$$f^{(\ve)}_r = \sup\left\{z \in\R\colon 
f-z+\ve\Phi_r\in \QM{S}_{2r}\right\}.$$

\begin{corollary}\label{c:numeric}
Let $S,f,\ve$ be as above. Then $(f^{(\ve)}_r)_r$ is an increasing sequence, and
\begin{equation}\label{e:poor1}
f_*\le\lim_{r\to \infty}f^{(\ve)}_r
\le\inf_{\uX\in K(S)}\left(
f(\uX)+\ve(\exp(X_1^2)+\cdots+\exp(X_n^2)\right).
\end{equation}
In particular, 
\begin{equation}\label{e:poor2}
\lim_{\ve\downarrow0}\lim_{r\to \infty}f^{(\ve)}_r = f_*.
\end{equation}
\end{corollary}

\begin{proof}
The first inequality in \eqref{e:poor1} holds by Corollary \ref{c:LN}, and the second inequality in \eqref{e:poor1} is straightforward. Lastly, \eqref{e:poor2} follows from
$\lim_{\ve\downarrow 0} \inf_K(f+\ve g) = \inf_Kf$ for any nonnegative function $g$ on $\R^n$.
\end{proof}

\bibliographystyle{alpha}
\bibliography{mompop}

\newcommand{\etalchar}[1]{$^{#1}$}
\begin{thebibliography}{TPKLR22}

\bibitem[AJK15]{alpay2015}
Daniel Alpay, Palle E.~T. Jorgensen, and David~P. Kimsey.
\newblock Moment problems in an infinite number of variables.
\newblock {\em Infinite Dimensional Analysis, Quantum Probability and Related
  Topics}, 18(04):1550024, 2015.

\bibitem[Are46]{arens}
Richard Arens.
\newblock The space {$L^\omega$} and convex topological rings.
\newblock {\em Bull. Amer. Math. Soc.}, 52:931--935, 1946.

\bibitem[Bar02]{Bar}
Alexander Barvinok.
\newblock {\em A course in convexity}, volume~54 of {\em Graduate Studies in
  Mathematics}.
\newblock American Mathematical Society, Providence, RI, 2002.

\bibitem[BBLM05]{boucheron05}
St{\'e}phane Boucheron, Olivier Bousquet, G{\'a}bor Lugosi, and Pascal Massart.
\newblock Moment inequalities for functions of independent random variables.
\newblock {\em The Annals of Probability}, 33(2):514 -- 560, 2005.

\bibitem[BKP16]{burgdorf16}
Sabine Burgdorf, Igor Klep, and Janez Povh.
\newblock {\em Optimization of polynomials in non-commuting variables}.
\newblock SpringerBriefs in Mathematics. Springer, 2016.

\bibitem[BP05]{BP05}
Dimitris Bertsimas and Ioana Popescu.
\newblock Optimal inequalities in probability theory: A convex optimization
  approach.
\newblock {\em SIAM Journal on Optimization}, 15(3):780--804, 2005.

\bibitem[BRS{\etalchar{+}}22]{blekherman2022}
Grigoriy Blekherman, Felipe Rincón, Rainer Sinn, Cynthia Vinzant, and
  Josephine Yu.
\newblock Moments, sums of squares, and tropicalization.
\newblock {\em arXiv:2203.06291}, 2022.

\bibitem[CG22]{nobel}
Davide Castelvecchi and Elizabeth Gibney.
\newblock ‘{S}pooky’ quantum-entanglement experiments win physics nobel.
\newblock {\em Nature}, 610:241--242, 2022.

\bibitem[CP10]{curto2010}
Ra{\'u}l Curto and Mihai Putinar.
\newblock Polynomially hyponormal operators.
\newblock In {\em A Glimpse at Hilbert Space Operators: Paul R. Halmos in
  Memoriam}, pages 195--207. Springer Basel, 2010.

\bibitem[DLTW08]{doherty2008quantum}
Andrew~C. Doherty, Yeong-Cherng Liang, Ben Toner, and Stephanie Wehner.
\newblock The quantum moment problem and bounds on entangled multi-prover
  games.
\newblock In {\em 2008 23rd Annual IEEE Conference on Computational
  Complexity}, pages 199--210. IEEE, 2008.

\bibitem[Fan22]{fantuzzi}
Giovanni Fantuzzi.
\newblock Verification of some functional inequalities via polynomial
  optimization.
\newblock {\em IFAC-PapersOnLine}, 55(16):166--171, 2022.
\newblock 18th IFAC Workshop on Control Applications of Optimization CAO 2022.

\bibitem[HIKV23]{henrion23}
Didier Henrion, Maria Infusino, Salma Kuhlmann, and Victor Vinnikov.
\newblock {Infinite-dimensional moment-SOS hierarchy for nonlinear partial
  differential equations}.
\newblock {\em arXiv:2305.18768}, 2023.

\bibitem[HKL20]{henrion2020moment}
Didier Henrion, Milan Korda, and Jean-Bernard Lasserre.
\newblock {\em Moment-sos Hierarchy, The: Lectures In Probability, Statistics,
  Computational Geometry, Control And Nonlinear Pdes}, volume~4.
\newblock World Scientific, 2020.

\bibitem[HM04]{Helton04}
J.~William Helton and Scott~A. McCullough.
\newblock A {P}ositivstellensatz for non-commutative polynomials.
\newblock {\em Trans. Amer. Math. Soc.}, 356(9):3721--3737, 2004.

\bibitem[IKKM23]{infusino2023moment}
Maria Infusino, Salma Kuhlmann, Tobias Kuna, and Patrick Michalski.
\newblock Moment problem for algebras generated by a nuclear space.
\newblock {\em arXiv:2301.12949}, 2023.

\bibitem[IKR14]{infusino2014}
Maria Infusino, Tobias Kuna, and Aldo Rota.
\newblock The full infinite dimensional moment problem on semi-algebraic sets
  of generalized functions.
\newblock {\em Journal of Functional Analysis}, 267(5):1382--1418, 2014.

\bibitem[KMV22]{KMV}
Igor Klep, Victor Magron, and Jurij Vol\v{c}i\v{c}.
\newblock Optimization over trace polynomials.
\newblock {\em Ann. Henri Poincar\'{e}}, 23(1):67--100, 2022.

\bibitem[KMVW23]{KMVW}
Igor Klep, Victor Magron, Jurij Vol\v{c}i\v{c}, and Jie Wang.
\newblock State polynomials: positivity, optimization and nonlinear bell
  inequalities.
\newblock {\em Math. Program.}, 2023.

\bibitem[KPT21]{kline21}
Brendan Kline, Ariel Pakes, and Elie Tamer.
\newblock Moment inequalities and partial identification in industrial
  organization.
\newblock In Kate Ho, Ali Horta\c{c}su, and Alessandro Lizzeri, editors, {\em
  Handbook of Industrial Organization}, volume~4, pages 345--431. Elsevier,
  2021.

\bibitem[KPV21]{KPV}
Igor Klep, James~E. Pascoe, and Jurij Vol\v{c}i\v{c}.
\newblock Positive univariate trace polynomials.
\newblock {\em J. Algebra}, 579:303--317, 2021.

\bibitem[K{\v{S}}V18]{klep2018positive}
Igor Klep, \v{S}pela {\v{S}}penko, and Jurij Vol\v{c}i\v{c}.
\newblock Positive trace polynomials and the universal {P}rocesi--{S}chacher
  conjecture.
\newblock {\em Proc. Lond. Math. Soc.}, 117(6):1101--1134, 2018.

\bibitem[KSV22]{KSV}
Igor Klep, Claus Scheiderer, and Jurij Vol\v{c}i\v{c}.
\newblock Globally trace-positive noncommutative polynomials and the unbounded
  tracial moment problem.
\newblock {\em Math. Ann.}, 387:1403--1433, 2022.

\bibitem[Las01]{Las01sos}
Jean-Bernard Lasserre.
\newblock Global optimization with polynomials and the problem of moments.
\newblock {\em SIAM J. Optim.}, 11(3):796--817, 2000/01.

\bibitem[Las06]{las06perturb}
Jean-Bernard Lasserre.
\newblock A sum of squares approximation of nonnegative polynomials.
\newblock {\em SIAM Journal on Optimization}, 16(3):751--765, 2006.

\bibitem[Las09]{lasserre2009moments}
Jean-Bernard Lasserre.
\newblock {\em {Moments, Positive Polynomials and Their Applications}}.
\newblock Imperial College Press optimization series. Imperial College Press,
  2009.

\bibitem[LJ10]{liao10}
Yuan Liao and Wenxin Jiang.
\newblock Bayesian analysis in moment inequality models.
\newblock {\em The Annals of Statistics}, 38(1):275--316, 2010.

\bibitem[LN07]{LasserreNetzer}
Jean-Bernard Lasserre and Tim Netzer.
\newblock Sos approximations of nonnegative polynomials via simple high degree
  perturbations.
\newblock {\em Math. Z.}, 256:1432--1823, 2007.

\bibitem[Mar08]{marshallbook}
Murray Marshall.
\newblock {\em Positive polynomials and sums of squares}, volume 146 of {\em
  Mathematical Surveys and Monographs}.
\newblock American Mathematical Society, Providence, RI, 2008.

\bibitem[MJC{\etalchar{+}}14]{mackey14}
Lester Mackey, Michael~I. Jordan, Richard~Y. Chen, Brendan Farrell, and Joel~A.
  Tropp.
\newblock {Matrix concentration inequalities via the method of exchangeable
  pairs}.
\newblock {\em The Annals of Probability}, 42(3):906 -- 945, 2014.

\bibitem[MW23]{sparsebook}
Victor Magron and Jie Wang.
\newblock {\em Sparse polynomial optimization: theory and practice}, volume~5
  of {\em Series on Optimization and Its Applications}.
\newblock World Scientific Press, 2023.

\bibitem[NPA08]{navascues2008convergent}
Miguel Navascu{\'e}s, Stefano Pironio, and Antonio Ac{\'\i}n.
\newblock A convergent hierarchy of semidefinite programs characterizing the
  set of quantum correlations.
\newblock {\em New J. Phys.}, 10(7):073013, 2008.

\bibitem[Par05]{parthasarathy2005}
Kalyanapuram~Rangachari Parthasarathy.
\newblock {\em Probability measures on metric spaces}, volume 352.
\newblock American Mathematical Soc., Providence, RI, 2005.

\bibitem[PHBB17]{PHBB}
Victor Pozsgay, Flavien Hirsch, Cyril Branciard, and Nicolas Brunner.
\newblock Covariance {B}ell inequalities.
\newblock {\em Phys. Rev. A}, 96(6):062128, 13, 2017.

\bibitem[PPHI15]{pakes15}
Ariel Pakes, Jack Porter, Kate Ho, and Joy Ishii.
\newblock Moment inequalities and their application.
\newblock {\em Econometrica}, 83(1):315--334, 2015.

\bibitem[Pro76]{P76}
Claudio Procesi.
\newblock The invariant theory of {$n\times n$} matrices.
\newblock {\em Adv. Math.}, 19(3):306--381, 1976.

\bibitem[PS01]{powers01}
Victoria Powers and Claus Scheiderer.
\newblock The moment problem for non-compact semialgebraic sets.
\newblock {\em Adv. Geom.}, 1(1):71--88, 2001.

\bibitem[Put93]{Putinar1993positive}
Mihai Putinar.
\newblock Positive polynomials on compact semi-algebraic sets.
\newblock {\em Indiana Univ. Math. J.}, 42(3):969--984, 1993.

\bibitem[Put97]{tchakaloff}
Mihai Putinar.
\newblock A note on {T}chakaloff's theorem.
\newblock {\em Proc. Amer. Math. Soc.}, 125(8):2409--2414, 1997.

\bibitem[Rud87]{rudin}
Walter Rudin.
\newblock {\em Real and complex analysis}.
\newblock McGraw-Hill Book Company, 1987.
\newblock 3rd edition.

\bibitem[Sch88]{schm88}
Konrad Schm{\"u}dgen.
\newblock Strongly commuting selfadjoint operators and commutants of unbounded
  operator algebras.
\newblock {\em Proc. Amer. Math. Soc.}, 102:365--372, 1988.

\bibitem[Sch90]{schmOalg}
Konrad Schm{\"u}dgen.
\newblock {\em Unbounded Operator Algebras and Representation Theory},
  volume~37 of {\em Operator Theory: Advances and Applications}.
\newblock Birkh{\"a}user Basel, 1990.

\bibitem[Sch12]{schmUnbded}
Konrad Schm{\"u}dgen.
\newblock {\em Unbounded self-adjoint operators on Hilbert space}, volume 265
  of {\em Graduate Texts in Mathematics}.
\newblock Springer Dordrecht, 2012.

\bibitem[Sch17]{schmbook}
Konrad Schm{\"u}dgen.
\newblock {\em The moment problem}, volume 277 of {\em Graduate Texts in
  Mathematics}.
\newblock Springer Cham, 2017.

\bibitem[SW99]{schaefer}
H.~H. Schaefer and M.~P. Wolff.
\newblock {\em Topological vector spaces}, volume~3 of {\em Graduate Texts in
  Mathematics}.
\newblock Springer-Verlag, New York, second edition, 1999.

\bibitem[TGB21]{TGB21}
Armin Tavakoli, Nicolas Gisin, and Cyril Branciard.
\newblock Bilocal {B}ell inequalities violated by the quantum elegant joint
  measurement.
\newblock {\em Phys. Rev. Lett.}, 126(22):Paper No. 220401, 2021.

\bibitem[TPKLR22]{tavakoli22}
Armin Tavakoli, Alejandro Pozas-Kerstjens, Ming-Xing Luo, and Marc-Olivier
  Renou.
\newblock Bell nonlocality in networks.
\newblock {\em Rep. Progr. Phys.}, 85(5):Paper No. 056001, 41, 2022.

\end{thebibliography}
\end{document}